\documentclass[11pt]{amsart}
\usepackage{amsmath,amssymb}
\usepackage{graphicx}
\usepackage{psfrag}
\usepackage{enumitem}
\usepackage{hyperref}

\usepackage{tikz,epsfig,color,graphicx,epstopdf,subfig}
\usetikzlibrary{positioning,chains,fit,shapes,calc, fit}
\date{} 
\def\COMMENT#1{}
\def\TASK#1{}
\begin{document}
\numberwithin{equation}{section}
\def\noproof{{\unskip\nobreak\hfill\penalty50\hskip2em\hbox{}\nobreak\hfill%
        $\square$\parfillskip=0pt\finalhyphendemerits=0\par}\goodbreak}
\def\endproof{\noproof\bigskip}
\newdimen\margin   
\def\textno#1&#2\par{%
    \margin=\hsize
    \advance\margin by -4\parindent
           \setbox1=\hbox{\sl#1}%
    \ifdim\wd1 < \margin
       $$\box1\eqno#2$$%
    \else
       \bigbreak
       \hbox to \hsize{\indent$\vcenter{\advance\hsize by -3\parindent
       \sl\noindent#1}\hfil#2$}%
       \bigbreak
    \fi}
\def\proof{\removelastskip\penalty55\medskip\noindent{\bf Proof. }}
\def\C{\mathcal{C}}
\def\eps{\varepsilon}
\def\epszero{\eps_0}
\def\ex{\mathbb{E}}
\def\prob{\mathbb{P}}
\def\eul{{\rm e}}
\def\cP{\mathcal{P}}
\def\cM{\mathcal{M}}
\newtheorem{firstthm}{Proposition}[section]
\newtheorem{thm}[firstthm]{Theorem}
\newtheorem{prop}[firstthm]{Proposition}
\newtheorem{fact}[firstthm]{Fact}
\newtheorem{lemma}[firstthm]{Lemma}
\newtheorem{cor}[firstthm]{Corollary}
\newtheorem{problem}[firstthm]{Problem}
\newtheorem{defin}[firstthm]{Definition}
\newtheorem{conj}[firstthm]{Conjecture}
\newtheorem{claim}[firstthm]{Claim}\def\i{(i_1,i_2,i_3,i_4)}
\def\I{i_1,i_2,i_3,i_4}
\def\Ifive{i_1,i_2,i_3,i_4,i_5}
\def\Ip{i'_1,i'_2,i'_3,i'_4}
\def\Ipfive{i'_1,i'_2,i'_3,i'_4,i'_5}
\title[Proof of the $1$-factorization \& Hamilton decomposition conjectures III]{Proof of the $1$-factorization and Hamilton decomposition conjectures III: approximate decompositions}
\author{B\'ela Csaba, Daniela K\"uhn, Allan Lo, Deryk Osthus and Andrew Treglown}
\thanks {The research leading to these results was partially supported by the  European Research Council
under the European Union's Seventh Framework Programme (FP/2007--2013) / ERC Grant
Agreement no. 258345 (B.~Csaba, D.~K\"uhn and A.~Lo), 306349 (D.~Osthus) and 259385 (A.~Treglown).
The research was also partially supported by the EPSRC, grant no. EP/J008087/1 (D.~K\"uhn and D.~Osthus).
}
\date{\today} 

\maketitle
\begin{abstract}
In a sequence of four papers, we prove the following results (via a unified approach) for all sufficiently large $n$:
\begin{itemize}
\item[(i)] [\emph{$1$-factorization conjecture}]
Suppose that $n$ is even and $D\geq 2\lceil n/4\rceil -1$. 
Then every $D$-regular graph $G$ on $n$ vertices has a decomposition into perfect matchings.
Equivalently, $\chi'(G)=D$.

\item[(ii)] [\emph{Hamilton decomposition conjecture}]
Suppose that $D \ge   \lfloor n/2 \rfloor $.
Then every $D$-regular graph $G$ on $n$ vertices has a decomposition
into Hamilton cycles and at most one perfect matching.

\item[(iii)] We prove an optimal result on the number of edge-disjoint Hamilton cycles in a graph 
of given minimum degree.
\end{itemize}
According to Dirac, (i) was first raised in the 1950s.
(ii) and (iii)
answer questions of Nash-Williams from 1970.
The above bounds are best possible. In the current paper, we show the following:
suppose that $G$ is close to a complete balanced bipartite graph or to the union of two cliques of equal size.
If we are given a suitable set of path systems which cover a set of `exceptional' vertices and edges of $G$, 
then we can extend these path systems into an approximate decomposition of $G$ into Hamilton cycles
(or perfect matchings if appropriate).
\end{abstract}

\maketitle

\section{Introduction}

\subsection{Background and results}

In a sequence of four papers, we develop a unified approach to prove the following results on Hamilton decompositions and 
$1$-factorizations. 
The first of these results confirms the so-called $1$-factorization conjecture for all sufficiently large graphs.
(A \emph{$1$-factorization} of a graph~$G$ consists of a set of edge-disjoint perfect matchings covering all edges of~$G$.)
This conjecture was first stated explicitly by Chetwynd and Hilton~\cite{1factorization,CH}.
However, they wrote that according to Dirac, it was already discussed in the 1950s.%
\COMMENT{Andy: changed 1950's to 1950s (same in abstract).}

\begin{thm}\label{1factthm}
There exists an $n_0 \in \mathbb N$ such that the following holds.
Let $ n,D \in \mathbb N$ be such that $n\geq n_0$ is even and $D\geq 2\lceil n/4\rceil -1$. 
Then every $D$-regular graph $G$ on $n$ vertices has a $1$-factorization.%
    \COMMENT{So this means that $D\ge n/2$ if $n = 2 \pmod 4$ and $D\ge n/2-1$ if $n = 0 \pmod 4$.}
    Equivalently, $\chi'(G)=D$.
\end{thm}
The bound on the degree in Theorem~\ref{1factthm} is best possible.
Nash-Williams~\cite{initconj,decompconj} raised the related problem of finding a Hamilton decomposition 
in an even-regular graph. Here
a decomposition  of an (even-regular) graph~$G$ into Hamilton cycles consists of a set of edge-disjoint Hamilton cycles covering all edges of~$G$.
If $G$ is a regular graph of odd degree, it is natural to ask for a perfect matching in $G$ together with a decomposition of the remaining edges into Hamilton cycles.
\begin{thm} \label{HCDthm} 
There exists an $n_0 \in \mathbb N$ such that the following holds.
Let $ n,D \in \mathbb N$ be such that $n \geq n_0$ and
$D \ge   \lfloor n/2 \rfloor $.
Then every $D$-regular graph $G$ on $n$ vertices has a decomposition into Hamilton cycles and 
at most one perfect matching.
\end{thm}
Again, the bound on the degree in Theorem~\ref{HCDthm} is best possible. In particular, the theorem settles the problem of 
Nash-Williams for all sufficiently large graphs.

Finally (in combination with~\cite{KLOmindeg}), we also  prove an optimal result on the number of edge-disjoint Hamilton cycles one can guarantee
in a graph of given minimum degree, which (as a special case) answers another question of Nash-Williams.
For a detailed discussion of the results and their background we refer to~\cite{paper1}.

\subsection{Overall structure of the argument}

For all  of our main results, we split the argument according to the structure of the graph $G$ under consideration:
\begin{enumerate}
\item[{\rm (i)}] $G$ is  close to the complete balanced bipartite graph $K_{n/2,n/2}$;
\item[{\rm (ii)}] $G$ is close to the union of two disjoint copies of a clique $K_{n/2}$;
\item[{\rm (iii)}] $G$ is a `robust expander'.
\end{enumerate}
Roughly speaking, $G$ is a robust expander if for every set $S$ of vertices, the neighbourhood of $S$ is at least a little larger than $|S|$,
even if we delete a small proportion of the vertices%
\COMMENT{Andy: added "vertices and"}
 and edges of $G$ (see Section~\ref{sec:rob}).
The main result of~\cite{Kelly} states that every dense regular robust expander
has a Hamilton decomposition.
This immediately implies Theorems~\ref{1factthm} and~\ref{HCDthm} in Case~(iii).
Case~(i) is proved in~\cite{paper2}. 
Case~(ii) is proved in~\cite{paper4,paper1}.
Moreover,~\cite{paper1} also includes a more detailed discussion of the overall structure of the proof.

\subsection{Contribution of the current paper} \label{sketch}
The arguments in~\cite{paper2,paper1}
make use of an `approximate decomposition' result, which  is proved in the current paper:
suppose that $G$ is as in Case (i) or (ii).
Suppose also that we are given a suitable set of path systems which cover a certain set of `exceptional' vertices and edges of $G$, 
then we can extend these path systems into an approximate decomposition of $G$ into Hamilton cycles
(or perfect matchings if appropriate).
The precise statements are given in Lemma~\ref{almostthmbip} for Case (i) and Lemma~\ref{almostthm} for Case (ii).

Roughly speaking, the strategy for the proof of Lemma~\ref{almostthm} is as follows.
By definition, $G$ contains disjoint sets $A$ and $B$ of size about $n/2$ which induce two almost complete graphs.
First we reduce the problem of finding an approximate decomposition of $G$ to that of finding approximate Hamilton decompositions of suitable graphs $G_A^*$ and $G_B^*$.
(Here $G_A^*$ is an almost complete multigraph with vertex set $A$ which contains $G[A]$, and $G_B^*$ is defined similarly.)
Moreover, each Hamilton cycle~$C$ in this approximate decomposition is required to contain a certain path system~$P$.
We find these Hamilton cycles by first extending $P$ into a path system which is `balanced' with respect to a given blown-up cycle.
This balanced path system is extended into a $1$-factor $F$ using edges which wind around the blown-up cycle.
Finally, $F$ is transformed into a Hamilton cycle $C$ using some edges which were set aside earlier.
(In particular, most edges of each Hamilton cycle in our approximate decomposition wind around a blown-up cycle.)%
    \COMMENT{Deryk: new sentence}
A more detailed sketch is given in Section~\ref{sec:sketch}.
The argument for the bipartite case~(i) is more involved but uses similar ideas
(in fact, some intermediate results are used in both cases).

This paper is organized as follows.
After introducing some notation, we state the main results (Lemmas~\ref{almostthm} and~\ref{almostthmbip})
in Section~\ref{state}. After introducing some tools in Section~\ref{tools},
we prove Lemma~\ref{almostthm} in Sections~\ref{systembalanced}--\ref{sec:extendmerge}.
We then prove Lemma~\ref{almostthmbip} using a similar approach in the final section.

\section{Notation}
The digraphs considered in this paper do not contain loops and we allow at most two edges between any pair of distinct vertices,
at most one edge in each direction.
Given a graph or digraph $G$, we write $V(G)$ for its vertex set, $E(G)$ for its edge set, $e(G):=|E(G)|$ for the number of edges in $G$ and $|G|:=|V(G)|$ for the number of vertices in $G$.

If $G$ is a graph and $v$ is a vertex of $G$, we write $N_G(v)$ for the set of all neighbours of~$v$ in~$G$. If $A\subseteq V(G)$,
we write $d_G(v,A)$ for the number of  neighbours of $v$ in $G$ which lie in~$A$.
We write $\delta(G)$ for the minimum degree of $G$, $\Delta(G)$ for its maximum degree
and $\chi'(G)$ for the edge-chromatic number of~$G$. Given $A,B\subseteq V(G)$,
we write $e_G(A)$ for the number of  edges of $G$ which have both endvertices in $A$ and $e_G(A,B)$ for the number of
\emph{$AB$-edges} of $G$, i.e.~for the number of  edges of $G$ which have one endvertex in $A$ and its other endvertex in $B$.
If $A\cap B=\emptyset$, we denote by $G[A,B]$ the bipartite subgraph of $G$
whose vertex classes are $A$ and $B$ and whose edges are all $AB$-edges of $G$. We often omit the index $G$ if the graph $G$ is clear from the context.
An \emph{$AB$-path} is a path having one endvertex in $A$ and its other endvertex in~$B$.
We often view a matching $M$ as a graph (in which every vertex has degree precisely one).%
    \COMMENT{This is different to eg the bipartite paper where a matching us a set of edges. We always use this
def in this paper, ie whether we write $e(M)$ for the number of edges and not $|M|$.}

If $G$ is a digraph, we write $xy$ for an edge directed from $x$ to~$y$. If $xy\in E(G)$, we say that $y$ is an
\emph{outneighbour} of~$x$ and $x$ is an \emph{inneighbour} of~$y$.
We write $d^+_G(x)$ for the \emph{outdegree} of $x$ (i.e.~for the number of outneighbours of~$x$ in~$G$)
and $d^-_G(x)$ for the \emph{indegree} of $x$ (i.e.~for the number of inneighbours of $x$).
We write $d^+_G(x,A)$ for the number of%
\COMMENT{Andy: deleted `all'} 
outneighbours of $x$ lying inside $A$ and define $d^-_G(x,A)$ similarly.
We denote the minimum outdegree of $G$ by $\delta^+(G)$ and the minimum indegree by $\delta^-(G)$.
Given $A,B\subseteq V(G)$, an \emph{$AB$-edge} is an edge with initial vertex in $A$ and final vertex in $B$,
and $e_G(A,B)$ denotes the number of these edges in~$G$. If $A\cap B=\emptyset$, we denote by $G[A,B]$ the bipartite subdigraph of $G$
whose vertex classes are $A$ and $B$ and whose edges are all $AB$-edges of $G$.
By a bipartite digraph $G=G[A,B]$ we mean a digraph which only contains $AB$-edges.%
    \COMMENT{This really means that we don't allow $BA$-edges}
A digraph $H$ is a \emph{$1$-factor} of another digraph $G$ if $V(H)=V(G)$ and every vertex in $H$ has in- and outdegree precisely one in~$H$.

Given a vertex set $V$ and two multigraphs%
   \COMMENT{Have to define this for multigraphs instead of graph since if we take $G+H+F$ then $G+H$ might already be a multigraph.
Also, $G\cup H$ is not defined, since we can get away with using $G+H$ in this paper}
$G$ and $H$ with $V(G),V(H)\subseteq V$, we write $G+H$ for the multigraph whose vertex
set is $V(G)\cup V(H)$ and in which the multiplicity of $xy$ in $G+H$ is the sum of the multiplicities of $xy$ in $G$ and in~$H$
(for all $x,y\in V(G)\cup V(H)$). Similarly, if $\mathcal{H}:=\{H_1,\dots,H_\ell\}$ is a set of graphs, we define
$G+\mathcal{H}:=G+H_1+\dots+H_\ell$. We write $G-H$ for the subgraph of $G$ which is obtained from $G$
by deleting all the edges in $E(G)\cap E(H)$.%
    \COMMENT{So we don't require that $H\subseteq G$ when using this notation. $G-A$ does not seem to be used, so is not defined} 
We say that a graph $G$ has a \emph{decomposition} into $H_1,\dots,H_r$ if $G=H_1+\dots +H_r$ and the $H_i$ are pairwise
edge-disjoint.

A \emph{path system} is a graph $Q$ which is the union of vertex-disjoint paths (some of them might be trivial).
We say that $P$ is a \emph{path in Q} if $P$ is a component of $Q$.
A \emph{path sequence} is a digraph which is the union of vertex-disjoint directed paths (some of them might be trivial).

Let $V_1,\dots,V_k$ be pairwise disjoint sets of vertices and let $C=V_1\dots V_k$ be a directed cycle on these sets.
We say that an edge $xy$ of a digraph $R$ \emph{winds around $C$} if there is some $i$ such that $x\in V_i$ and $y\in V_{i+1}$.
In particular, we say that $R$ \emph{winds around $C$} if all edges of $R$ wind around~$C$.

In order to simplify the presentation, we omit floors and ceilings and treat large numbers as integers whenever this does
not affect the argument. The constants in the hierarchies used to state our results have to be chosen from right to left.
More precisely, if we claim that a result holds whenever $0<1/n\ll a\ll b\ll c\le 1$ (where $n$ is the order of the graph),
then this means that
there are non-decreasing functions $f:(0,1]\to (0,1]$, $g:(0,1]\to (0,1]$ and $h:(0,1]\to (0,1]$ such that the result holds
for all $0<a,b,c\le 1$ and all $n\in \mathbb{N}$ with $b\le f(c)$, $a\le g(b)$ and $1/n\le h(a)$. 
We will not calculate these functions explicitly. Hierarchies with more constants are defined in a similar way.
Given $a,b,c\in\mathbb{R}$, we will write $a = b \pm c$ as shorthand for $ b - c \le a \le b+c$.

\section{Statements of the main results} \label{state}
\subsection{Statement for the two cliques case}
In this section, we state our approximate decomposition result in the case when $G$ is a graph which is close to
the union of two disjoint copies of $K_{n/2}$.
Suppose that $A,A_0,B,B_0$ forms a partition of a vertex set $V$ of size $n$ such that $|A| = |B|$. Let $V_0:=A_0\cup B_0$. Let $A' := A_0 \cup A$ and $B' := B_0 \cup B$.
$V_0$ will be a small set which consists of `exceptional' vertices.
Below, we will define an `exceptional cover', which is a path system which covers the vertices in~$V_0$.
Our aim is to extend such exceptional covers into edge-disjoint Hamilton cycles (or pairs of edge-disjoint perfect matchings).
As an intermediate step, we will consider `exceptional systems',%
   \COMMENT{Daniela: changed `exceptional path systems' to `exceptional systems'}
which are exceptional covers that also contain an appropriate number of
`connections' between $A'$ and $B'$.

Our main result (Lemma~\ref{almostthm}) then states the following: suppose that we are given a graph $G$ so that $G[A]$ and $G[B]$ are almost complete
and that we are given an appropriate set of edge-disjoint exceptional systems. Then we can extend these exceptional systems into a set of edge-disjoint Hamilton cycles
(and possibly perfect matchings) of $G$
covering almost all edges of $G[A]$ and $G[B]$. Moreover, the additional edges are all contained in $G[A]+G[B]$.

More precisely, an \emph{exceptional cover} $J$ is a graph which satisfies the following properties:
\begin{enumerate}[label={(EC{\arabic*})}]
\item $J$ is a path system with $V_0\subseteq V(J)\subseteq V$.
\item $d_J(v) =2 $ for every $v \in V_0$ and $d_J(v) \le 1$ for every $v \in V(J) \setminus V_0$.
\item $e_J(A), e_J(B) = 0$.
\end{enumerate}
We say that $J$ is an \emph{exceptional system with parameter~$\eps_0$}, or an \emph{ES} for short, if $J$ satisfies the following properties:
\begin{enumerate}[label={(ES{\arabic*})}]
	\item $J$ is an exceptional cover.
	\item One of the following is satisfied:
	\begin{itemize}
	\item[(HES)] The number of $AB$-paths in $J$ is even and positive. In this case we say $J$ is a \emph{Hamilton exceptional system}, or \emph{HES} for short.
	\item[(MES)] $e_J(A',B')=0$. In this case we say $J$ is a \emph{matching exceptional system}, or \emph{MES} for short. 
\end{itemize}
	\item $J$ contains at most $\sqrt{\eps_0} n $ $AB$-paths.
\end{enumerate}
Note that by~(EC2) every $AB$-path in $J$ must be a maximal path in~$J$. Moreover, the number of $AB$-paths in $J$ is the number of genuine `connections' between $A$ and $B$
(and thus between $A'$ and $B'$).
If we want to extend $J$ into a Hamilton cycle using only edges induced by $A$ and edges induced by~$B$,
this number clearly has to be even and positive.
Hamilton exceptional systems will always be extended into Hamilton cycles and matching exceptional systems will always be extended into two disjoint even cycles which together span all vertices (and thus consist of two edge-disjoint perfect matchings).

We will need to consider exceptional systems which are `localized' in the sense that their vertices are contained in a small number of
clusters of a given partition. To formalize this, let $K,m\in\mathbb{N}$ and $\eps_0>0$.
A \emph{$(K,m,\eps_0)$-partition $\mathcal{P}$} of a set $V$ of vertices is a partition of $V$ into sets $A_0,A_1,\dots,A_K$
and $B_0,B_1,\dots,B_K$ such that $|A_i|=|B_i|=m$ for all $i\ge 1$ and $|A_0\cup B_0|\le \eps_0 |V|$.
The sets $A_1,\dots,A_K$ and $B_1,\dots,B_K$ are called \emph{clusters} of $\mathcal{P}$
and $A_0$, $B_0$ are called \emph{exceptional sets}. We often write $V_0$ for $A_0\cup B_0$ and think of the
vertices in $V_0$ as `exceptional vertices'. Unless stated otherwise, whenever $\mathcal{P}$ is a $(K,m,\eps_0)$-partition,
we will denote the clusters by $A_1,\dots,A_K$ and $B_1,\dots,B_K$ and the exceptional sets by $A_0$ and $B_0$.
We will also write $A:=A_1\cup\dots\cup A_K$, $B:=B_1\cup\dots\cup B_K$, $A':=A_0\cup A$
and $B':=B_0\cup B$.

Given a $(K,m, \epszero)$-partition $\mathcal{P}$ and $1\le i,i' \le K$, we say that $J$ is an 
 \emph{$ (i,i')$-localized exceptional system} with respect to $\mathcal{P}$
(abbreviated as \emph{$(i,i')$-ES}) if $J$ is an  exceptional system and $V(J)\subseteq V_0 \cup A_{i} \cup B_{i'}$.%
\COMMENT{We don't use $(i,i')$-HES nor $(i,i')$-MES, so they are not defined.}

\begin{lemma}\label{almostthm}
Suppose that $0<1/n \ll \eps_0  \ll 1/K \ll \rho  \ll 1$ and $0 \le \mu \ll 1$,
where $n,K \in \mathbb N$ and $K$ is odd.%
	\COMMENT{previously had $1/K \leq \mu$, but this seems clearer. When reading, need to make sure all estimates involving $\mu$ are ok now}
Suppose that $G$ is a graph on $n$ vertices and $\mathcal{P}$ is a $(K, m, \eps _0)$-partition of $V(G)$.
Furthermore, suppose that the following conditions hold:%
	\COMMENT{Previously we had `$(G[A]+G[B],\mathcal{P})$ is a $(K, m, \eps _0, \eps )$-scheme', but this is not neeeded.}
\begin{itemize}
	\item[{\rm (a)}] $d(v,A_i) = (1 - 4 \mu \pm 4 /K) m $ and $d(w,B_i) = (1 - 4 \mu \pm 4 /K) m $ for all
	$v \in A$, $w \in B$ and $1\leq i \leq K$.
	\item[{\rm (b)}] There is a set $\mathcal J$ which consists of at most $(1/4-\mu - \rho)n$ edge-disjoint exceptional systems with parameter $\eps_0$ in~$G$.%
	\COMMENT{By the definition of exceptional system, $J \subseteq G - G[A] - G[B]$ for all $J\in \mathcal{J}$. So we can omit this
condition (which we had previously).}
	\item[{\rm (c)}] $\mathcal J$ has a partition into $K^2$ sets $\mathcal J_{i,i'}$ (one for all $1\le i,i'\le K$) such that each $\mathcal J_{i,i'}$ consists of precisely $|\mathcal J|/{K^2}$ $(i,i')$-ES with respect to~$\cP$.
   \item[{\rm (d)}] If $\mathcal{J}$ contains matching exceptional systems then $|A'|=|B'|$ is even.
\end{itemize}
Then $G$ contains $|\mathcal J|$ edge-disjoint spanning subgraphs $H_1,\dots,H_{|\mathcal J|}$ which satisfy the following
properties:
\begin{itemize}
\item For each $H_s$ there is some $J_s\in \mathcal{J}$ such that $J_s\subseteq H_s$.
\item If $J_s$ is a Hamilton exceptional system, then $H_s$ is a Hamilton cycle of $G$. If
$J_s$ is a matching exceptional system, then $H_s$ is the edge-disjoint union of two perfect matchings in $G$.
\end{itemize}
\end{lemma}
\subsection{Statement for the bipartite case} \label{mainbi}
We now state an analogous result for the case when our graph $G$ is close to the complete balanced bipartite graph $K_{n/2,n/2}$.
Let $\cP:=\{A_0,A_1,\dots,A_K,B_0,B_1,\dots,B_K\}$ be a $(K,m,\eps_0)$-partition of a set $V$ of $n$ vertices.
Define $A,A',B,B'$ as in the previous subsection.
We now define the analogue of an $(i,i')$-localized exceptional system for the bipartite case as follows.

Given $1\leq \I \leq K$ and $\eps_0>0$, an \emph{$\i$-balanced exceptional system with respect to $\cP$ and parameter~$\eps_0$}, or \emph{$\i$-BES} for short, 
is a path system $J$ with $V(J)\subseteq A_0\cup B_0\cup A_{i_1} \cup A_{i_2} \cup B_{i_3} \cup B_{i_4}$
such that the following conditions hold:
\begin{itemize}
\item[(BES$1$)] $d_J(v)=2$ for every vertex $v\in V_0=A_0\cup B_0$ and $d_J(v)\le 1$ for
every vertex $v\in A_{i_1} \cup A_{i_2} \cup B_{i_3} \cup B_{i_4}$.%
   \COMMENT{earlier the paths had length at most 2. But current def seems to be ok. Also, note that
the BES we construct in bip paper have the property that none of these paths has both endpoints
in $A_{i_1} \cup A_{i_2}$ and its midoint in $A_0$ (and the analogue holds with for $B$).}   
\item[(BES$2$)] 
Every edge of $J[A \cup B]$ is either an $A_{i_1}A_{i_2}$-edge or a $B_{i_3}B_{i_4}$-edge.%
	\COMMENT{Allan: had `Every (maximal) path of length one in~$J$ has either one endpoint in $A_{i_1}$ and the other endpoint in $A_{i_2}$
or one endpoint in $B_{i_3}$ and the other endpoint in $B_{i_4}$.'}
\item[(BES$3$)] The edges in $J$ cover precisely the same number of vertices in $A$ as in $B$.%
   \COMMENT{Cannot write $J$ instead of "edges in $J$" here since we don't want to count the vertices of degree 0 which lie in $V(J)$.}
\item[(BES4)] $e(J)\le \eps_0 n$.
\end{itemize}
Note that (BES2) implies that an $(i_1,i_2, i_3,i_4)$-BES does not contain $AB$-edges.
Furthermore, an $\i$-BES is also, for example, an $(i_2,i_1,i_4,i_3)$-BES. We sometimes omit the indices $i_1,i_2,i_3,i_4$ and just refer to
a balanced exceptional system.

We can now state the analogue of Lemma~\ref{almostthm} for the case when $G$ is close to $K_{n/2,n/2}$.

\begin{lemma}\label{almostthmbip}
Suppose that $0<1/n \ll \eps_0  \ll 1/K \ll \rho  \ll 1$ and $0 \leq \mu \ll 1$,
where $n,K \in \mathbb N$ and $K$ is even.
Suppose that $G$ is a graph on $n$ vertices and $\mathcal{P}$ is a $(K, m, \eps _0)$-partition of $V(G)$.
Furthermore, suppose that the following conditions hold:
\begin{itemize}
	\item[{\rm (a)}] $d(v,B_i) = (1 - 4 \mu \pm 4 /K) m $ and $d(w,A_i) = (1 - 4 \mu \pm 4 /K) m $ for all
	$v \in A$, $w \in B$ and $1\leq i \leq K$.
	\item[{\rm (b)}] There is a set $\mathcal J$ which consists of at most $(1/4-\mu - \rho)n$ edge-disjoint balanced exceptional systems with parameter $\eps_0$ in~$G$.
	\item[{\rm (c)}] $\mathcal J$ has a partition into $K^4$ sets $\mathcal J_{i_1,i_2,i_3,i_4}$ (one for all $1\le \I \le K$) such that each $\mathcal J_{\I}$ consists of precisely $|\mathcal J|/{K^4}$ $\i$-BES with respect to~$\cP$.
   \item[{\rm (d)}] For each $v \in A \cup B$ the number of $J \in \mathcal{J}$ such that $v$ is incident with an edge of $J$ is at most $2 \eps_0 n $.%
   \COMMENT{This is a new property. This condition is implied by the fact that $(G[A,B],\mathcal{P})$ is a $(K, m, \eps _0, \eps )$-scheme.
($|A_0 \cup B_0| \le \eps_0 n$ and $\Delta( G[A]),\Delta( G[B]) \le \eps_0 n$.)}
\end{itemize}
Then $G$ contains $|\mathcal J|$ edge-disjoint Hamilton cycles such that each of these Hamilton cycles contains some $J\in \mathcal J$.
\end{lemma}


\section{Useful results} \label{tools}
\subsection{A Chernoff estimate}
We will use the following Chernoff bound for the binomial 
distribution (see e.g.~\cite[Corollary 2.3]{Janson&Luczak&Rucinski00}).
Recall that a binomial random variable with parameters $(n,p)$ is the sum
of $n$ independent Bernoulli variables, each taking value $1$ with probability $p$
or $0$ with probability $1-p$.

\begin{prop}\label{chernoff}
Suppose $X$ has binomial distribution and $0<a<3/2$. Then
$\mathbb{P}(|X - \mathbb{E}X| \ge a\mathbb{E}X) \le 2 e^{-\frac{a^2}{3}\mathbb{E}X}$.
\end{prop}

\subsection{Regular spanning subgraphs}
The following lemma implies that any almost complete balanced bipartite graph has an approximate decomposition into perfect matchings.
The proof is a straightforward application of the MaxFlowMinCut theorem.

\begin{lemma}\label{regularsub}
Suppose that
$0<1/m \ll \eps \ll \rho \ll 1$, that $0 \le \mu \le 1/4$ and that $m, \mu m, \rho m \in \mathbb{N}$.
Suppose that $\Gamma$ is a bipartite graph with vertex classes $U$ and $V$ of size $m$ and with
$(1-\mu-\eps)m \leq \delta (\Gamma) \leq \Delta (\Gamma) \leq (1-\mu+\eps)m$.  Then $\Gamma$ contains a spanning $(1-\mu-\rho)m$-regular subgraph
$\Gamma '$.
In particular, $\Gamma$ contains at least $(1-\mu-\rho)m$ edge-disjoint perfect matchings.
\end{lemma}
\proof 
We first obtain a directed network $N$ from $\Gamma$ by adding a source $s$ and a sink $t$.
We add a directed edge $su$ of capacity $(1-\mu-\rho)m$ for each $u \in U$ and a directed edge $vt$ of capacity 
$(1-\mu-\rho)m$ for each $v \in V$. We give all the edges in $\Gamma$ capacity $1$ and direct them from $U$ to $V$.

Our aim is to show that the capacity of any $(s,t)$-cut%
	\COMMENT{Allan: replace cut with $(s,t)$-cut.}
 is at least $(1-\mu-\rho)m^2$. By the MaxFlowMinCut theorem this would imply
that $N$ admits an integer-valued flow of value $(1-\mu-\rho)m^2$ which by construction of $N$ implies the existence of our desired subgraph $\Gamma '$.

Consider any $(s,t)$-cut $(S, \overline{S})$ where $S= \{s\} \cup S_1 \cup S_2$ with $S_1 \subseteq U$ and $S_2 \subseteq V$.
Let $\overline{S}_1:=U\backslash S_1$ and $\overline{S}_2 :=V \backslash S_2$.
The capacity of this cut is
$$(1-\mu-\rho)m(m-|S_1|)+e(S_1, \overline{S}_2)+ (1-\mu-\rho)m|S_2|$$
and therefore our aim is to show that
\begin{align}\label{target}
e(S_1, \overline{S}_2) \geq (1-\mu-\rho)m(|S_1|-|S_2|).
\end{align}
If $|S_1| \le (1-\mu-\rho )m$, then 
\begin{align*}
	e(S_1, \overline{S}_2) &  \ge 
\left( ( 1-\mu-\eps) m - |S_2| \right) |S_1| \\
		& = (1-\mu-\rho) m (|S_1|-|S_2|) + (\rho-\eps ) m |S_1| +|S_2|\left( (1-\mu-\rho)m -|S_1| \right)\\
		& \ge  (1-\mu-\rho) m (|S_1|-|S_2|).
\end{align*}
Thus, we may assume that $|S_1| > (1-\mu-\rho )m$.
Note that $|S_1|-|S_2| = |\overline{S}_2|-|\overline{S}_1|$.
Therefore, by a similar argument, we may also assume that $|\overline{S}_2| > (1-\mu-\rho)m$ and so $|S_2| \le (\mu +\rho) m$.
This implies that
\begin{align*}
	e(S_1, \overline{S}_2) & \ge \sum_{x \in S_1} d_{\Gamma} (x) - \sum_{y \in S_2} d_{\Gamma} (y)
 \ge  (1-\mu-\eps ) m |S_1| - (1-\mu+\eps )m |S_2| \\
	& =(1-\mu-\rho) m (|S_1| - |S_2|) + \rho m (|S_1| -|S_2|) - \eps m( |S_1| +|S_2|)  \\
	& > (1-\mu-\rho)m (|S_1|-|S_2|) + (1- 2\mu-2\rho) \rho m^2 - (1+\mu+\rho) \eps m^2\\
	& \ge  (1-\mu-\rho)m (|S_1|-|S_2|).
\end{align*}
(Note that the last inequality follows as $\eps \ll \rho\ll 1$ and $ \mu \le 1/4$.)
So indeed (\ref{target}) is satisfied, as desired.
\endproof


\subsection{Robust expansion} \label{sec:rob}

Given $0<\nu \leq \tau<1$, we say that a digraph $G$ on $n$ vertices is a \emph{robust $(\nu, \tau)$-outexpander},
if for all $S\subseteq V(G)$ with $\tau n\le |S|\le (1-\tau)n$ the number of vertices that have at least $\nu n$
inneighbours in $S$ is at least $|S|+\nu n$. 
The following result was derived in~\cite{3reg3con} as a straightforward consequence of the
result from~\cite{KOT10} that every robust outexpander of linear minimum degree has a Hamilton cycle.%
    \COMMENT{Daniela: reworded}

\begin{thm}\label{expanderthm}
Suppose that $0 < 1/n\ll \gamma \ll \nu \ll \tau\ll\eta<1$. Let~$G$ be a digraph on~$n$ vertices with
$\delta^+(G), \delta^-(G)\ge \eta n$ which is a robust $(\nu,\tau)$-outexpander. 
Let $y_1, \dots, y_p$ be distinct vertices in $V(G)$ with $p \le \gamma n$.
Then~$G$ contains a directed Hamilton cycle visiting $y_1, \dots, y_p$ in this  order.
\end{thm}


\subsection{A regularity concept for sparse graphs} \label{sec:reg}
We now formulate a concept of $\eps$-superregularity which is suitable for `sparse' graphs.
Let $G$ be a bipartite graph with vertex classes $U$ and $V$, both of size $m$.
Given $A\subseteq U$ and $B\subseteq V$, we write $d(A,B):=e(A,B)/|A||B|$ for the density of $G$
between $A$ and~$B$. Given $0<\eps,d,d^*,c<1$, we say that $G$ is \emph{$(\eps,d,d^*,c)$-superregular} if the following conditions are satisfied:
\begin{itemize}
\item[(Reg1)] Whenever $A\subseteq U$ and $B\subseteq V$ are sets of size at least $\eps m$, then
$d(A,B)=(1\pm \eps)d$.
\item[(Reg2)] For all $u,u'\in V(G)$ we have $|N(u)\cap N(u')|\le c^2m$.%
	\COMMENT{Previously, we had the following statement, which is the same.
For all $u,u'\in U$ we have $|N(u)\cap N(u')|\le c^2m$. Similarly, for all $v,v'\in V$ we have $|N(v)\cap N(v')|\le c^2m$.}
\item[(Reg3)] $\Delta(G)\le cm$.
\item[(Reg4)] $\delta(G)\ge d^*m$.
\end{itemize}
Note that the above definitions also make sense if $G$ is `sparse' in the sense that $d<\eps$ (which will be the case in our proofs).
A bipartite digraph $G=G[U,V]$ is \emph{$(\eps,d,d^*,c)$-superregular} if this holds for
the underlying undirected graph of~$G$.

The following observation follows immediately from the definition.

\begin{prop} \label{superslice6}
Suppose that $0<1/m \ll  d^*,d, \eps, \eps', c  \ll 1$ and $2\eps'\le d^*$.
Let $G$ be an $(\eps,d,d^*,c)$-superregular bipartite graph with vertex classes
$U$ and $V$ of size $m$. 
Let $U' \subseteq U$ and $V' \subseteq V$ with $|U'|=|V'| \ge (1- \eps')m$.
Then $G[U',V']$ is $(2\eps,d, d^*/2 ,2c)$-superregular.%
   \COMMENT{Let $m':= |U'| = |V'|$, so $(1- \eps')m \le m' \le m$.
Since $ 2\eps m'\ge \frac{\eps m'}{1- \eps'} \ge \eps m$, $G[U',V']$ satisfies (Reg1)--(Reg3).
To see (Reg4) note that the degrees in $G'$ are still at least $d^*m-\eps' m \ge d^* m'/2$.}
\end{prop}

The following two simple observations were made in~\cite{Kelly}.
\begin{prop} \label{superslice5}
Suppose that $0<1/m \ll d^*,d, \eps, c\ll 1$. Let $G$ be an $(\eps,d,d^*,c)$-superregular bipartite graph with vertex classes
$U$ and $V$ of size $m$. Suppose that $G'$ is obtained from $G$ by removing
at most $\eps^2 dm$ edges incident to each vertex from $G$.
Then $G'$ is $(2\eps,d,d^*-\eps^2 d,c)$-superregular.
\end{prop}

\begin{lemma}\label{regtoexpander}
Let $0<1/m\ll \nu\ll \tau\ll d\le \eps\ll \mu,\zeta \le 1/2$ and let $G$ be an
$(\eps,d,\zeta d,d/\mu)$-superregular bipartite graph with vertex classes $U$ and $V$ of size $m$. Let $A\subseteq U$ be such that
$\tau m\le |A|\le (1-\tau) m$. Let $B\subseteq V$ be the set of all those vertices in $V$ which have at least $\nu m$ neighbours in $A$.
Then $|B|\ge |A|+\nu m$.
\end{lemma}


\section{Systems and Balanced extensions} \label{systembalanced}
\subsection{Sketch proof of Lemma~\ref{almostthm}} \label{sec:sketch}
Roughly%
   \COMMENT{Deryk: new para}
speaking, the Hamilton cycles we find will have the following structure:
let $A_1,\dots,A_K\subseteq A$ and $B_1,\dots,B_K\subseteq B$ be the clusters of the $(K,m, \epszero)$-partition $\mathcal{P}$
of $V(G)$ given in Lemma~\ref{almostthm}. So $K$ is odd.
Let $\mathcal{R}_A$ be the complete graph on $A_1,\dots,A_K$ and $\mathcal{R}_B$ be the complete graph on $B_1,\dots,B_K$.
Since $K$ is odd, Walecki's theorem~\cite{lucas} implies that $\mathcal{R}_A$ has a Hamilton decomposition $C_{A,1}, \dots, C_{A,(K-1)/2}$, and similarly  $\mathcal{R}_B$
has a Hamilton decomposition $C_{B,1}, \dots, C_{B,(K-1)/2}$.%
\COMMENT{Andy: added last part of sentence.}
Every Hamilton cycle $C$ we construct in $G$ will have the property that there is a $j$ so that almost all edges of $C[A]$
wind around $C_{A,j}$ and almost all edges of $C[B]$ wind around $C_{B,j}$.
Below, we describe the main ideas involved in the construction of the Hamilton cycles in more detail.

As indicated above, the first idea is that we can reduce the problem of finding the required edge-disjoint Hamilton cycles (and possibly perfect matchings) in~$G$
to that of finding appropriate Hamilton cycles on each of $A$ and $B$ separately.
We achieve this by introducing suitable path systems $J_A^*$ and $J_B^*$.

More precisely, let $\mathcal{J}$ be a set of edge-disjoint exceptional systems as given in Lemma~\ref{almostthm}.%
    \COMMENT{Daniela: reworded}
By deleting some edges if necessary, we may further assume that $\mathcal{J}$ is an edge-decomposition of $G - G[A] - G[B]$.
Thus, in order to prove Lemma~\ref{almostthm}, we have to find $|\mathcal{J}|$ suitable edge-disjoint subgraphs $H_{A,1}, \dots, H_{A,|\mathcal{J}|}$
of $G[A]$ and $|\mathcal{J}|$ suitable edge-disjoint subgraphs $H_{B,1}, \dots, H_{B,|\mathcal{J}|}$ of $G[B]$ such that
$H_s := H_{A,s} + H_{B,s}  + J_s$ are the desired spanning subgraphs of~$G$.
To prove this, for each $J \in \mathcal{J}$, we define two auxiliary subgraphs $J^*_A$ and $J^*_B$ with the following crucial properties: 
\begin{itemize}
\item[($\alpha_1$)] $J^*_A$ and $J^*_B$ are matchings whose vertices are contained in $A$ and $B,$\COMMENT{B\'ela: added comma} respectively;
\item[($\alpha_2$)] the union of any Hamilton cycle $C^*_A$ in $G[A] + J^*_A$ containing $J^*_A$ (in some suitable order)
and any Hamilton cycle $C^*_B$ in $G[B] + J^*_B$ containing $J^*_B$ (in some suitable order) corresponds to either a Hamilton cycle
of $G$ containing $J$ or to the union of two edge-disjoint perfect matchings of $G$ containing $J$.
\end{itemize}
Furthermore, $J$ determines which of the cases in~($\alpha_2$) holds: If $J$ is a Hamilton exceptional system,
then ($\alpha_2$) will give a Hamilton cycle of $G$, while in the case when $J$ is a matching exceptional system, ($\alpha_2$)
will give the union of two edge-disjoint perfect matchings of $G$.
So roughly speaking, this allows us to work with multigraphs $G^*_A := G[A] + \sum_{J \in \mathcal{J}} J^*_A$ and
$G^*_B:= G[B] + \sum_{J \in \mathcal{J} } J^*_B$ rather than $G$ in the two steps.%
	\COMMENT{Allan: added the word multigraphs}
Furthermore, the processes of finding Hamilton cycles in $G^*_A$ and in $G^*_B$ are independent (see Section~\ref{sec:J*2clique} for more details).%
   \COMMENT{Deryk: added brackets} 

By symmetry, it suffices to consider $G^*_A$ in what follows.
The second idea of the proof is that as an intermediate step, we decompose $G_A^*$ into blown-up Hamilton cycles
$G^*_{A,j}$. Roughly speaking, we will then find an approximate Hamilton decomposition of each $G^*_{A,j}$ separately.

More precisely,%
    \COMMENT{Andy: reworded} 
recall that $\mathcal{R}_A$ denotes the complete graph whose vertex set is $\{A_1, \dots, A_K \}$.
As mentioned above, $\mathcal{R}_A$ has a Hamilton decomposition $C_{A,1}, \dots, C_{A,(K-1)/2}$.%
    \COMMENT{Daniela: changed $C_j$ to $C_{A,j}$ here and below}
We decompose $G[A]$ into edge-disjoint subgraphs $G_{A,1}, \dots,G_{A,{(K-1)/2}}$ such that each $G_{A,j}$ corresponds to the `blow-up' of $C_{A,j}$,
i.e.~$G_{A,j}[U,W] = G[U,W]$ for every edge $UW \in E(C_{A,j})$. (The edges of $G$ lying inside one of the clusters $A_1, \dots, A_K$
are deleted.) We also partition the set $\{J^*_A: J\in\mathcal{J}\}$ into $(K-1)/2$ sets
$\mathcal{J}^*_{A,1}, \dots, \mathcal{J}^*_{A,(K-1)/2}$ of roughly equal size. Set $G^*_{A,j} := G_{A,j} + \mathcal{J}^*_{A,j}$.
Thus in order to prove Lemma~\ref{almostthm}, we need to find $|\mathcal{J}^*_{A,j}|$ edge-disjoint Hamilton cycles in $G^*_{A,j}$
(for each $j\le (K-1)/2$).
Since $G^*_{A,j}$ is still close to being a blow-up of the cycle $C_{A,j}$, finding such Hamilton cycles 
seems feasible.

One complication is that in order to satisfy~($\alpha_2$), we need to ensure that each Hamilton cycle in $G^*_{A,j}$
contains some $J^*_A\in \mathcal{J}^*_{A,j}$ (and it must traverse the edges of $J^*_A$ in some given order).
To achieve this, we will both orient and order the edges of~$J^*_A$.
So we will actually consider an ordered directed matching $J^*_{A, {\rm dir}}$ instead of $J^*_A$. ($J^*_A$ itself will still be undirected and unordered).
We orient the edges of $G_{A,j}$ such that the resulting oriented graph $G_{A,j,{\rm dir}}$ is a blow-up of the directed cycle $C_{A,j}$.

However, $J^*_{A,{\rm dir}}$ may not be `locally balanced with respect to $C_{A,j}$'.
This means that it is impossible to extend $J^*_{A,{\rm dir}}$ into a directed Hamilton cycle using only edges of $G_{A,j,{\rm dir}}$.
For example, suppose that $G_{A,j,{\rm dir}}$ is a blow-up of the directed cycle $A_1A_2 \dots A_{K}$, i.e.~each edge of $G_{A,j,{\rm dir}}$
joins $A_i$ to $A_{i+1}$ for some $1\le i\le K$.
If $J^*_{A,{\rm dir}}$ is non-empty and $V(J^*_{A,{\rm dir}})\subseteq A_1$, then $J^*_{A,{\rm dir}}$
cannot be extended into a directed Hamilton cycle using edges of $G_{A,j,{\rm dir}}$ only.
Therefore, each $J^*_{A,{\rm dir}} $ will first be extended into a `locally balanced path sequence' $PS$.
$PS$ will have the property that it can be extended to a Hamilton cycle
using only edges of $G_{A,j,{\rm dir}}$. We will call the set $\mathcal{BE}_j$ consisting of all such $PS$ for all $J^*_A\in \mathcal{J}^*_{A,j}$
a balanced extension of $\mathcal{J}^*_{A,j}$.%
   \COMMENT{Deryk: changed $\mathcal{J}^*_{j}$ to $\mathcal{J}^*_{A,j}$ twice}
$\mathcal{BE}_j$ will be constructed in
Section~\ref{sec:BE2clique}, using edges from a sparse graph $H'$ on $A$ (which is actually removed from $G[A]$ before defining $G_{A,1}, \dots, G_{A,(K-1)/2}$).

Finally, we find the required directed Hamilton cycles in $G_{A,j,{\rm dir}}+ \mathcal{BE}_j$ in Section~\ref{sec:extendmerge}.
We construct these by first extending the path sequences in $\mathcal{BE}_j$ into (directed) $1$-factors, using edges which wind around the blow-up of $C_{A,j}$.
These are then transformed into Hamilton cycles using a small set of edges set aside earlier (again the set of these edges winds around the blow-up of $C_{A,j}$).

\subsection{Systems and balanced extensions} \label{system}
As mentioned above, the proof of Lemma~\ref{almostthm} requires an edge-decomposition and orientation of $G[A]$ and
$G[B]$ into blow-ups of directed cycles as well as `balanced extensions'. These are defined in the current subsection.

Let $k,m \in \mathbb{N}$.
A \emph{$(k,m)$-equipartition} $\mathcal{Q}$ of a set $V$ of vertices is a partition of $V$ into sets $V_1, \dots, V_k$ such that $|V_i| = m$ for all $i \le k$.
The $V_i$ are called \emph{clusters} of~$\mathcal{Q}$.
$(G, \mathcal Q, C)$ is a \emph{$(k,m, \mu, \eps )$-cyclic system}
 if the following properties hold:
\begin{enumerate}[label={(Sys{\arabic*})}]
	\item $G$ is a digraph and $\mathcal{Q}$ is a $(k,m)$-equipartition of $V(G)$.
	\item $C$ is a directed Hamilton cycle on $\mathcal{Q}$ and $G$ winds around $C$.%
    \COMMENT{Added the "winding around bit" since we need it in the proof of Lemma~\ref{merging}}
Moreover, for every edge $UW$ of $C$,
we have $d^+_G(u,W) = (1- \mu \pm \eps)m $ for every $u \in U$ and $d^-_G(w,U) = (1- \mu \pm \eps)m$ for every $w \in W$.
\end{enumerate}
So roughly speaking, such a cyclic system is a blown-up Hamilton cycle.

Let $\mathcal{Q}$ be a \emph{$(k,m)$-equipartition} of $V$ and let $C$ be a directed Hamilton cycle on~$\mathcal{Q}$.
We say that a digraph $H$ with $V(H)\subseteq V$ is \emph{locally balanced with respect to $C$} if for every edge $UW$
of $C$, the number of edges of $H$ with initial vertex in $U$ equals the number of edges of $H$ with final vertex in~$W$.

Let $M$ be a directed matching.
We say a path sequence $PS$ is a \emph{$V_i$-extension of $M$ with respect to $\mathcal{Q}$} if each edge of $M$ is
contained in a distinct directed path in $PS$ having its final vertex in~$V_i$.%
	\COMMENT{We write $V_i$-extension instead of $i$-extension, because in the bipartite case, we take $\mathcal{Q} = \{A_1, \dots, A_K, B_1, \dots, B_K\}$.}
Let $\mathcal{M} := \{ M_{1}, \dots, M_{q}\}$ be a set of directed matchings. 
A set $\mathcal{BE}$ of path sequences is a \emph{balanced extension of $\mathcal{M}$ with respect to $(\mathcal{Q},C)$
and parameters $(\eps, \ell)$} if $\mathcal{BE}$ satisfies the following properties:
\begin{itemize}
	\item[(BE1)] $\mathcal{BE}$ consists of $q$ path sequences $PS_1, \dots, PS_q$ such that $V(PS_i)\subseteq V$ for
each $i\le q$, each $PS_i$ is locally balanced with respect to $C$ and $PS_1-M_1,\dots,PS_q-M_q$ are edge-disjoint from each other.%
   \COMMENT{Daniela: now have that $PS_1-M_1,\dots,PS_q-M_q$ are edge-disjoint instead of the $PS_s$ themselves}
	\item[(BE2)] Each $PS_s$ is a $V_{i_s}$-extension of $M_s$ with respect to $\mathcal{Q}$ for some $i_s \le k$.
Moreover, for each $i \le k$ there are at most $\ell m/k$  indices $s\le q$ such that $i_s=i$.%
    \COMMENT{will have $k=2K$ in the bip case}
	\item[(BE3)] $|V(PS_s) \cap V_i |\le \eps m$ for all $i \le k $ and $s \le q$. Moreover, for each $i \le k$, there are at most $\ell m/k$ path sequences $PS_s \in \mathcal{BE}$ such that $V(PS_s) \cap V_i \ne \emptyset$.
\end{itemize}

Note that the `moreover part' of (BE3) implies the `moreover part' of (BE2).%
     \COMMENT{Deryk: new sentence}

Given an ordered directed matching $M = \{f_1, \dots, f_{\ell}\}$, we say that a directed cycle $C'$ is
\emph{consistent with $M$} if $C'$ contains $M$ and visits the edges $f_1, \dots ,f_{\ell}$ in this order.
The following observation will be useful:
suppose that $PS$ is a $V_i$-extension of $M$ and let $x_j$ be the final vertex of the path in $PS$ containing $f_j$.
(So $x_1, \dots, x_{\ell}$ are distinct vertices of $V_i$.)
Suppose also that $C'$ is a directed cycle which contains $PS$  and visits $x_1, \dots, x_{\ell}$ in this order.
Then $C'$ is consistent with $M$.

\section{Finding systems and balanced extensions for the two cliques cases}
Let $G$ be a graph, let $\mathcal{P}$ be a $(K,m, \eps_0)$-partition of $V(G)$ and
let $\mathcal{J}$ be a set of exceptional systems as given by Lemma~\ref{almostthm}.
The aim of this section is to decompose $G[A]+G[B]$ into $(k, m, \mu , \eps )$-cyclic systems and to construct balanced extensions
as described in Section~\ref{sec:sketch}. First we need to define $J^*_{A, {\rm dir}}$ and $J^*_{B, {\rm dir}}$ for each exceptional system~$J\in \mathcal{J}$.
Recall from Section~\ref{sec:sketch} that these are introduced in order to be able to consider $G[A]$ and $G[B]$ separately
(and thus to be able to ignore the exceptional vertices in $V_0=A_0 \cup B_0$).

\subsection{Defining the graphs $J^*_{A,{\rm dir}}$ and $J^*_{B, {\rm dir}}$} \label{sec:J*2clique}

Let $A,A_0,B,B_0$ be a partition of a vertex set $V$ on $n$ vertices and let $J$ be an exceptional system with parameter $\eps_0$.
Since each maximal path in $J$ has endpoints in $A \cup B$ and internal vertices in $V_0$, an exceptional system $J$
naturally induces a matching $J^*_{AB}$ on $A \cup B$.
More precisely, if $P_1, \dots ,P_{\ell'}$ are the non-trivial paths in~$J$ and $x_i, y_i$ are the endpoints of $P_i$, then
we define $J^*_{AB} := \{x_iy_i : i  \le \ell'\}$. Thus $e_{J^*_{AB}}(A,B)$ is equal to the number of $AB$-paths in~$J$.
In particular, $e_{J^*_{AB}}(A,B)$ is positive and even if $J$ is a Hamilton exceptional system, while $e_{J^*_{AB}}(A,B)=0$ if $J$
is a matching exceptional system. Without loss of generality we may assume that $x_1y_1, \dots, x_{2\ell}y_{2\ell}$ is
an enumeration of the edges of $J^*_{AB}[A,B]$, where $x_i \in A$ and $y_i \in B$. Define 
$$
J_A^* := J^*_{AB}[A] \cup \{x_{2i-1} x_{2i} :1 \le i \le \ell \} \ \
 \mbox{and} \ \ \ J_B^* := J^*_{AB}[B] \cup \{y_{2i} y_{2i+1} :1 \le i \le \ell \}
$$ 
(with indices considered modulo $2\ell$).
Let $J^* := J_A^* + J_B^*$. So $J^*$ is a matching and $e(J^*) = e(J^*_{AB})$. Moreover, by (EC2), (EC3) and (ES3) we have
\begin{align}
	e(J^*) = e(J^*_{AB}) \le |V_0| + \sqrt{\eps_0} n. \label{ESeq}
\end{align}
We will call the edges in $J^*$ \emph{fictive} edges. Note that if $J_1$ and $J_2$ are two edge-disjoint exceptional systems, then $J^*_1$
and $J^*_2$ may not be edge-disjoint.%
   \COMMENT{Daniela: deleted the following, since it is not needed anymore: However, we will always view fictive edges as being distinct from each other and from the edges
in other graphs. So in particular, whenever $J_1$ and $J_2$ are two exceptional systems, we will view $J^*_1$
and $J^*_2$ as being edge-disjoint.}

We say that an (undirected) cycle $C$ is \emph{consistent with $J_A^*$} if $C$ contains $J_A^*$ and (there is an orientation of $C$ which)
visits the vertices $x_1, \dots ,x_{2\ell}$ in this order.%
    \COMMENT{We need to prescribe a vertex rather than an edge ordering as it is important in Proposition~\ref{prop:ES}.}
In a similar way we define when a cycle is consistent with $J_B^*$.
The following proposition is proved in~\cite{paper1}. It is illustrated in Figure~\ref{fig2}.

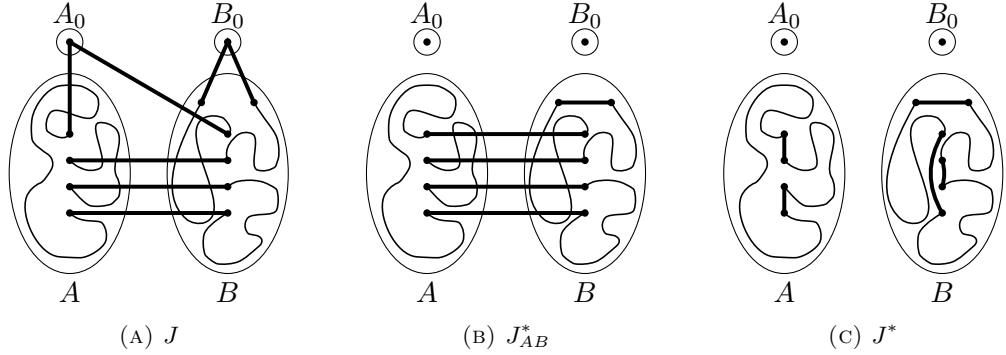
\begin{figure}[tbp]
\centering
\subfloat[$J$]{
\begin{tikzpicture}[scale=0.35]
			\draw  (-3,0) ellipse (2.3  and 3.8 );
			\draw (3,0) ellipse (2.3  and 3.8 );
			\node at (-3,-4.5) {$A$};
			\node at (3,-4.5) {$B$}; 			
			\draw (-3,5) circle(0.5);
			\draw (3,5) circle (0.5);
			\node at (-3,6) {$A_0$};
			\node at (3,6) {$B_0$};
			\fill (-3,5) circle (4pt);
			\fill (3,5) circle (4pt);
			\begin{scope}[start chain]
			\foreach \i in {1.5,0.5,...,-1.5}
			\fill (-3,\i) circle (4pt);
			\end{scope}
			\begin{scope}[start chain]
			\foreach \i in {1.5,0.5,...,-1.5}
			\fill (3,\i) circle (4pt);
			\end{scope}
			\fill (2,2.7) circle (4pt);
			\fill (4,2.7) circle (4pt);

			\begin{scope}[line width=1.5pt]
			\draw (-3,1.5)--(-3,5)--(3,1.5);
			\draw (2,2.7)--(3,5)--(4,2.7);
				\begin{scope}[start chain]
				\foreach \i in {0.5,-0.5,-1.5}
				\draw (-3,\i)--(3,\i);
				\end{scope}
			\end{scope}
			\begin{scope}[line width=0.6pt]
			\draw[rounded corners]  (-3,1.5) to [out=235, in=-70] (-4,2) to [out=50, in=-130] (-1.8,2.5) to [out=110, in=-10] (-3,3.4) to [out=-170, in=80] (-4.5,1.7) to [out=-80, in=145] (-3.5,0.5) to [out=-120, in=20] (-4.5,-0.5) 
 to [out=-110, in=175]  (-3,-3.2) to [out=0, in= -110] (-1.5,-2.2) to [out=130, in=-45] (-3,-1.5);		
 			\draw[rounded corners]  (-3,0.5) to [out=-45, in=150] (-2.2,0) to [out=30, in=-170] (-1.5,2) to [out=-50, in= 120] (-1,-0.5) to [out=-110, in=20] (-1.8,-1.3) to [out=170, in= -50] (-3,-0.5);
			 \draw[rounded corners]  (3,-0.5) to [out=40, in=-170] (4,-0.2) to [out=-20, in= 110] (5,-1) to [out=-110, in= 50] (3.5,-2) to [out=-45, in= 160] (4.5, -2.5) to [out=-140, in= 20] (3,-3.5) to [out=170, in= -70] (2,-2.7) to [out=75, in= -150] (3,-1.5);
			\draw[rounded corners] (2,2.7) to [out=-120, in= 140]  (1.75,-2) to [out=20, in= -155]  (2.5,2.3) to [out=-30, in= 30] (3,1.5);
			\draw[rounded corners] (4,2.7)  to [out=-60, in= 110]  (5,1)  to [out=-100, in= 40]  (4.5,0) to [out=160, in=-30] (3.6,1.7) to [out=-135, in=110]  (3,0.5);
			\end{scope}
\end{tikzpicture}
}
\qquad
\subfloat[$J_{AB}^*$]{
\begin{tikzpicture}[scale=0.35]
			\draw  (-3,0) ellipse (2.3  and 3.8 );
			\draw (3,0) ellipse (2.3  and 3.8 );
			\node at (-3,-4.5) {$A$};
			\node at (3,-4.5) {$B$}; 			
			\draw (-3,5) circle(0.5);
			\draw (3,5) circle (0.5);
			\node at (-3,6) {$A_0$};
			\node at (3,6) {$B_0$};
			\fill (-3,5) circle (4pt);
			\fill (3,5) circle (4pt);
			\begin{scope}[start chain]
			\foreach \i in {1.5,0.5,...,-1.5}
			\fill (-3,\i) circle (4pt);
			\end{scope}
			\begin{scope}[start chain]
			\foreach \i in {1.5,0.5,...,-1.5}
			\fill (3,\i) circle (4pt);
			\end{scope}
			\fill (2,2.7) circle (4pt);
			\fill (4,2.7) circle (4pt);

			\begin{scope}[line width=1.5pt]
			\draw (-3,1.5)--(3,1.5);
			\draw (2,2.7)--(4,2.7);
				\begin{scope}[start chain]
				\foreach \i in {0.5,-0.5,-1.5}
				\draw (-3,\i)--(3,\i);
				\end{scope}
			\end{scope}

			\begin{scope}[line width=0.6pt]
			\draw[rounded corners]  (-3,1.5) to [out=235, in=-70] (-4,2) to [out=50, in=-130] (-1.8,2.5) to [out=110, in=-10] (-3,3.4) to [out=-170, in=80] (-4.5,1.7) to [out=-80, in=145] (-3.5,0.5) to [out=-120, in=20] (-4.5,-0.5) 
 to [out=-110, in=175]  (-3,-3.2) to [out=0, in= -110] (-1.5,-2.2) to [out=130, in=-45] (-3,-1.5);		
 			\draw[rounded corners]  (-3,0.5) to [out=-45, in=150] (-2.2,0) to [out=30, in=-170] (-1.5,2) to [out=-50, in= 120] (-1,-0.5) to [out=-110, in=20] (-1.8,-1.3) to [out=170, in= -50] (-3,-0.5);
			 \draw[rounded corners]  (3,-0.5) to [out=40, in=-170] (4,-0.2) to [out=-20, in= 110] (5,-1) to [out=-110, in= 50] (3.5,-2) to [out=-45, in= 160] (4.5, -2.5) to [out=-140, in= 20] (3,-3.5) to [out=170, in= -70] (2,-2.7) to [out=75, in= -150] (3,-1.5);
			\draw[rounded corners] (2,2.7) to [out=-120, in= 140]  (1.75,-2) to [out=20, in= -155]  (2.5,2.3) to [out=-30, in= 30] (3,1.5);
			\draw[rounded corners] (4,2.7)  to [out=-60, in= 110]  (5,1)  to [out=-100, in= 40]  (4.5,0) to [out=160, in=-30] (3.6,1.7) to [out=-135, in=110]  (3,0.5);
				\end{scope}
\end{tikzpicture}
}
\qquad
\subfloat[$J^*$]{
\begin{tikzpicture}[scale=0.35]
			\draw  (-3,0) ellipse (2.3  and 3.8 );
			\draw (3,0) ellipse (2.3  and 3.8 );
			\node at (-3,-4.5) {$A$};
			\node at (3,-4.5) {$B$}; 			
			\draw (-3,5) circle(0.5);
			\draw (3,5) circle (0.5);
			\node at (-3,6) {$A_0$};
			\node at (3,6) {$B_0$};
			\fill (-3,5) circle (4pt);
			\fill (3,5) circle (4pt);
			\begin{scope}[start chain]
			\foreach \i in {1.5,0.5,...,-1.5}
			\fill (-3,\i) circle (4pt);
			\end{scope}
			\begin{scope}[start chain]
			\foreach \i in {1.5,0.5,...,-1.5}
			\fill (3,\i) circle (4pt);
			\end{scope}
			\fill (2,2.7) circle (4pt);
			\fill (4,2.7) circle (4pt);

			\begin{scope}[line width=1.5pt]
			\draw (-3,1.5)--(-3,0.5);
			\draw (-3,-1.5)--(-3,-0.5);
			\draw (2,2.7)--(4,2.7);
			\draw (3,0.5)to[out=-75, in=75](3,-0.5);
			\draw (3,1.5)to[out=-120, in=120](3,-1.5);
			\end{scope}

			\begin{scope}[line width=0.6pt]		
			\draw[rounded corners]  (-3,1.5) to [out=235, in=-70] (-4,2) to [out=50, in=-130] (-1.8,2.5) to [out=110, in=-10] (-3,3.4) to [out=-170, in=80] (-4.5,1.7) to [out=-80, in=145] (-3.5,0.5) to [out=-120, in=20] (-4.5,-0.5) 
 to [out=-110, in=175]  (-3,-3.2) to [out=0, in= -110] (-1.5,-2.2) to [out=130, in=-45] (-3,-1.5);		
 			\draw[rounded corners]  (-3,0.5) to [out=-45, in=150] (-2.2,0) to [out=30, in=-170] (-1.5,2) to [out=-50, in= 120] (-1,-0.5) to [out=-110, in=20] (-1.8,-1.3) to [out=170, in= -50] (-3,-0.5);
			 \draw[rounded corners]  (3,-0.5) to [out=40, in=-170] (4,-0.2) to [out=-20, in= 110] (5,-1) to [out=-110, in= 50] (3.5,-2) to [out=-45, in= 160] (4.5, -2.5) to [out=-140, in= 20] (3,-3.5) to [out=170, in= -70] (2,-2.7) to [out=75, in= -150] (3,-1.5);
			\draw[rounded corners] (2,2.7) to [out=-120, in= 140]  (1.75,-2) to [out=20, in= -155]  (2.5,2.3) to [out=-30, in= 30] (3,1.5);
			\draw[rounded corners] (4,2.7)  to [out=-60, in= 110]  (5,1)  to [out=-100, in= 40]  (4.5,0) to [out=160, in=-30] (3.6,1.7) to [out=-135, in=110]  (3,0.5);
			\end{scope}
\end{tikzpicture}
}
\caption{The thick lines illustrate the edges of $J$, $J_{AB}^*$ and $J^*$, respectively.}
\label{fig2}
\end{figure}

\begin{prop} \label{prop:ES}
Suppose that $A,A_0,B,B_0$ forms a partition of a vertex set $V$.
Let $J$ be an exceptional system. Let $C_A$ and $C_B$ be two cycles such that 
\begin{itemize}
	\item $C_A$ is a Hamilton cycle on $A$ that is consistent with $J_A^*$;
	\item $C_B$ is a Hamilton cycle on $B$ that is consistent with $J_B^*$.
\end{itemize}
Then the following assertions hold. 
\begin{itemize}
	\item[\rm (i)] If $J$ is a Hamilton exceptional system, then $C_A+C_B - J^* +J$ is a Hamilton cycle on $V$.
	\item[\rm (ii)] If $J$ is a  matching exceptional system, then $C_A+C_B - J^* +J$ is the union of a Hamilton cycle on $A'$ and a Hamilton cycle on $B'$.
In particular, if both $|A'|$ and $|B'|$ are even, then $C_A+C_B - J^* +J$ is the union of two edge-disjoint perfect matchings on $V$.
\end{itemize}
\end{prop}

As mentioned in Section~\ref{sec:sketch}, we will orient and order the edges of $J^*_A$ and $J^*_{B}$ in a suitable way
to obtain $J^*_{A, {\rm dir}}$ and $J^*_{B, {\rm dir}}$.
Accordingly, we will actually need an oriented version of Proposition~\ref{prop:ES}.
For this,  we first orient the edges of $J^*_{A}$ by orienting the edge $x_{2i-1} x_{2i}$ from $x_{2i-1}$ to $x_{2i}$ for all $i \le \ell$ and the edges of $J^*_{AB}[A]$ arbitrarily. 
Next we order these directed edges as $f_1, \dots, f_{\ell_A}$ such that $f_i = x_{2i-1} x_{2i}$ for all $i \le \ell$, where $\ell_A := e(J_A^*)$.
Define $J_{A, { \rm dir } }^*$ to be the ordered directed matching $\{f_1, \dots, f_{\ell_A} \}$.
Similarly, to define $J_{B, { \rm dir } }^*$, we first orient the edges of $J^*_{B}$ by orienting the edge $y_{2i} y_{2i+1}$
from $y_{2i}$ to $y_{2i+1}$ for all $i \le \ell$ and the edges of $J^*_{AB}[B]$ arbitrarily. 
Next we order these directed edges as $f_1', \dots, f_{\ell_B}'$ such that $f_i' = y_{2i} y_{2i+1}$ for all $i \le \ell$, where $\ell_B := e(J_B^*)$.
Define $J_{B, { \rm dir } }^*$ to be the ordered directed matching $\{f_1', \dots, f_{\ell_B}' \}$.
Note that if $J$ is an $(i,i')$-ES, then $V(J_{A, { \rm dir } }^*) \subseteq A_i$ and $V(J_{B, { \rm dir } }^*) \subseteq B_{i'}$.
Recall from Section~\ref{system} that a directed cycle $C_{A, {\rm dir}}$ is \emph{consistent with} $J_{A, { \rm dir } }^*$
if $C_{A, {\rm dir}}$ contains $J_{A, { \rm dir } }^*$ and visits the edges $f_1, \dots ,f_{\ell_A}$ in this order.%
   \COMMENT{Daniela: reworded} 
The following proposition follows easily from Proposition~\ref{prop:ES}.

\begin{prop} \label{prop:ES2}
Suppose that $A,A_0,B,B_0$ forms a partition of a vertex set $V$.
Let $J$ be an exceptional system. Let $C_{A, {\rm dir}}$ and $C_{B, {\rm dir}}$ be two directed cycles such that 
\begin{itemize}
	\item $C_{A, {\rm dir}}$ is a directed Hamilton cycle on $A$ that is consistent with $J_{A, {\rm dir}}^*$;
	\item $C_{B, {\rm dir}}$ is a directed Hamilton cycle on $B$ that is consistent with $J_{B, {\rm dir}}^*$.
\end{itemize}
Then the following assertions hold, where $C_A$ and $C_B$ are the undirected cycles obtained from $C_{A, {\rm dir}}$
and $C_{B, {\rm dir}}$ by ignoring the directions of all the edges. 
\begin{itemize}
	\item[\rm (i)] If $J$ is a Hamilton exceptional system, then $C_A+C_B - J^* +J$ is a Hamilton cycle on $V$.
		
	\item[\rm (ii)] If $J$ is a  matching exceptional system, then $C_A+C_B - J^* +J$ is the union of a Hamilton cycle on $A'$ and a Hamilton cycle on $B'$.
In particular, if both $|A'|$ and $|B'|$ are even, then $C_A+C_B - J^* +J$ is the union of two edge-disjoint perfect matchings on $V$.
\end{itemize}
\end{prop}

\subsection{Finding systems}
In this subsection, we will decompose (and orient) $G[A]$ into cyclic systems
$(G_{A,j,{\rm dir}}, \mathcal{Q}_A, C_{A,j})$, one for each $j\le (K-1)/2$.
Roughly speaking, this corresponds to a decomposition into (oriented) blown-up Hamilton cycles.
We will achieve this by considering a Hamilton decomposition of $\mathcal{R}_A$, where 
$\mathcal{R}_A$ is the complete graph on $\{A_1, \dots, A_K \}$. So each $C_{A,j}$ corresponds to one of the
Hamilton cycles in this Hamilton decomposition. We will split the set $\{J^*_{A,{\rm dir}} : J \in \mathcal{J} \}$
into subsets $\mathcal{J}^*_{A,j}$ and assign $\mathcal{J}^*_{A,j}$ to the $j$th cyclic system. Moreover,
for each $j\le (K-1)/2$,%
   \COMMENT{Daniela replaced $j\le (K-1)/2$ by $j\le (K-1)/2$}
we will also set aside a sparse spanning subgraph $H_{A,j}$%
\COMMENT{Andy:replaced $H_j$ with $H_{A,j}$.}
of $G[A]$, which is removed from $G[A]$
before the decomposition into cyclic systems. $H_{A,j}$%
\COMMENT{Andy:replaced $H_j$ with $H_{A,j}$.}
will be used later on in order to
find a balanced extension of $\mathcal{J}^*_{A,j}$. We proceed similarly for $G[B]$.

\begin{lemma}\label{sysdecom}
Suppose that $0<1/n \ll \eps_0  \ll 1/K \ll \rho    \ll 1$ and $0 \leq \mu \ll 1$,
where $n,K \in \mathbb N$ and $K$ is odd.
Suppose that $G$ is a graph on $n$ vertices and $\mathcal{P}$ is a $(K, m, \eps _0)$-partition of $V(G)$.
Furthermore, suppose that the following conditions hold:%
\begin{itemize}
	\item[{\rm (a)}] $d(v,A_i) = (1 - 4 \mu \pm 4 /K) m $ and $d(w,B_i) = (1 - 4 \mu \pm 4 /K) m $ for all
	$v \in A$, $w \in B$ and $1\leq i \leq K$.
	\item[{\rm (b)}] There is a set $\mathcal J$ which consists of at most $(1/4-\mu - \rho)n$ edge-disjoint exceptional systems with parameter $\eps_0$ in~$G$.
	\item[{\rm (c)}] $\mathcal J$ has a partition into $K^2$ sets $\mathcal J_{i,i'}$ (one for all $1\le i,i'\le K$) such that each
$\mathcal J_{i,i'}$ consists of precisely $|\mathcal J|/{K^2}$ $(i,i')$-ES with respect to~$\cP$.
   \item[{\rm (d)}] If $\mathcal{J}$ contains matching exceptional systems then $|A'|=|B'|$ is even.\end{itemize}%
Then for each $1 \le j \le (K-1)/2$, there is a pair of tuples $(G_{A,j}, \mathcal{Q}_{A}, C_{A,j}, H_{A,j}, \mathcal{J}^*_{A,j})$
and $(G_{B,j}, \mathcal{Q}_{B}, C_{B,j}, H_{B,j}, \mathcal{J}^*_{B,j})$  such that the following assertions hold:
\begin{itemize}
\item[{\rm (a$_1$)}] Each of $C_{A,1}, \dots, C_{A,(K-1)/2}$ is a directed Hamilton cycle on $\mathcal{Q}_A := \{A_1, \dots, A_K \}$
such that the undirected versions of these Hamilton cycles form a Hamilton decomposition of the complete graph on $\mathcal{Q}_A$.
\item[{\rm (a$_2$)}] $\mathcal{J}^*_{A,1}, \dots, \mathcal{J}^*_{A,(K-1)/2}$ is a partition of $\{J^*_{A,{\rm dir}} : J \in \mathcal{J} \}$.
\item[{\rm (a$_3$)}] Each $\mathcal{J}^*_{A,j}$ has a partition into $K$ sets $\mathcal J^*_{A,j,i}$ (one for each $1\le i\le K$)
such that $|\mathcal J^*_{A,j,i}| \le (1- 4\mu - 3\rho) m/K$ and 
each $J^*_{A, {\rm dir}} \in \mathcal J^*_{A,j,i}$ is an ordered directed matching with $e(J^*_{A, {\rm dir}})\le 5 K \sqrt{\eps_0} m$
and $V(J^*_{A, {\rm dir}})\subseteq A_i$.
\item[{\rm (a$_4$)}] $G_{A,1},\dots, G_{A,(K-1)/2}, H_{A,1},\dots,  H_{A,(K-1)/2}$ are edge-disjoint subgraphs of $G[A]$.
\item[{\rm (a$_5$)}] $H_{A,j}[A_{i},A_{i'}]$ is a $10K\sqrt{\eps_0}m$-regular graph for all $j \le (K-1)/2$ and all $i,i' \le K$ with $i \ne i'$.%
\COMMENT{We have omitted the floor/ceiling on $10K\sqrt{\eps_0}m$.}
\item[{\rm (a$_6$)}] For each $j \le (K-1)/2$, there exists an orientation $G_{A,j,{\rm dir}}$ of $G_{A,j}$ such
that $(G_{A,j,{\rm dir}}, \mathcal{Q}_A, C_{A,j})$ is a $(K,m, 4\mu, 5/K)$-cyclic system.
\item[{\rm (a$_7$)}] Analogous statements to {\rm(a$_1$)--(a$_6$)} hold for $C_{B,j}, \mathcal{J}^*_{B,j},G_{B,j}, H_{B,j}$ for all
$j \le (K-1)/2$, with $\mathcal{Q}_B := \{ B_1, \dots, B_K\}$.
\end{itemize}
\end{lemma}

\begin{proof}
Since $K$ is odd, by Walecki's theorem the complete graph on $\{A_1, \dots, A_K \}$ has a Hamilton decomposition.
(a$_1$) follows by orienting the edges of each of these Hamilton cycles to obtain directed Hamilton cycles $C_{A,1}, \dots, C_{A,(K-1)/2}$.

For each $i,i' \le K$, we partition $\mathcal J_{i,i'}$ into $(K-1)/2$ sets $\mathcal J_{i,i',j}$ (one for each $j\le (K-1)/2$)
whose sizes are as equal as possible.
Note that if $J \in \mathcal J_{i,i',j}$, then $J$ is an $(i,i')$-ES and so $V(J^*_{A, { \rm dir} }) \subseteq A_i$.
Since $\mathcal{P}$ is a $(K,m, \eps_0)$-partition of $V(G)$, $|V_0| \le \eps_0 n $ and $(1- \eps_0) n \le 2Km$.
Hence, 
\begin{align*}
e(J^*_{A, { \rm dir} }) \le e(J^*) \overset{\eqref{ESeq}}{\le} |V_0| + \sqrt{\eps_0}n 
\le 5 \sqrt{\eps_0} K m.
\end{align*}
(a$_2$) is satisfied by setting $\mathcal{J}^*_{A,j,i} := \bigcup_{i' \le K} \{J^*_{A, { \rm dir} }: J \in \mathcal J_{i,i',j}\} $
and $\mathcal{J}^*_{A,j} := \bigcup_{i \le K}  \mathcal{J}^*_{A,j,i}$.
Note that 
\begin{align*}
|\mathcal{J}^*_{A,j,i}| & \le \sum_{i' \le K}\left(\frac{2|\mathcal J_{i,i'}|}{K-1}+1 \right) \overset{{\rm (c)}}{=}  \frac{2|\mathcal J|}{K(K-1)}+K  
 \overset{{\rm (b)}}{\le} \frac{(1/2 - 2\mu - 2\rho) n}{K(K-1)} + K \\
& \le  (1-4 \mu - 3 \rho) m/K
\end{align*}
as $2K m  \ge (1- \eps_0)  n $ and $1 /n \ll \eps_0 \ll 1/K \ll \rho$.
Hence (a$_3$) holds.

For $i,i' \le K$ with $i \ne i'$, apply Lemma~\ref{regularsub} with $G[A_i,A_{i'}], 4/K, \rho, 4\mu$ playing the roles of $\Gamma, \eps, \rho, \mu$
to obtain a spanning $(1- 4 \mu - \rho)m$-regular subgraph $H_{i,i'}$ of $G[A_i,A_{i'}]$.
Since $H_{i,i'}$ is a regular bipartite graph and $\epszero \ll 1/K , \rho \ll 1$ and $0 \le \mu \ll 1$,
there exist $(K-1)/2$ edge-disjoint $10K \sqrt{ \epszero } m$-regular spanning subgraphs $H_{i,i',1}, \dots, H_{i,i',(K-1)/2}$ of $H_{i,i'}$.
Set $H_{A,j} := \sum_{1 \le i, i' \le K} H_{i,i',j}$ for each $j \le (K-1)/2$. So (a$_5$) holds.

Define $G_A:=G[A]-(H_{A,1}+\dots+H_{A,(K-1)/2})$.
Note that, as $\eps_0 \ll 1/K$, (a) implies that $d_{G_A}(v,A_i) = (1 - 4 \mu \pm 5 /K) m $ for all $v \in A$ and all $i \le K$.
For each $j \le (K-1)/2$, let $G_{A,j}$ be the graph on $A$ whose edge set is the union of $G_A[A_{i},A_{i'}]$ for each edge $A_iA_{i'} \in E(C_{A,j})$.
Define $G_{A,j, {\rm dir}}$ to be the oriented graph obtained from $G_{A,j}$ by orienting every edge in $G_A[A_{i},A_{i'}]$ from $A_i$ to $A_{i'}$
(for each edge $A_iA_{i'} \in E(C_{A,j})$).
Note that $ ( G_{A,j, {\rm dir}}, \mathcal{Q}_A, C_{A,j})$ is a $(K,m, 4\mu, 5/K) $-cyclic system for each $j \le (K-1)/2$.
Therefore, (a$_4$) and (a$_6$) hold. (a$_7$) can be proved by a similar argument.
\end{proof}

\subsection{Constructing balanced extensions} \label{sec:BE2clique}

Let $(G_{A,j}, \mathcal{Q}_A, C_{A,j}, H_{A,j} , \mathcal{J}^*_{A,j})$ be one of the $5$-tuples obtained by Lemma~\ref{sysdecom}.
The next lemma will be applied to find a balanced extension of $\mathcal{J}^*_{A,j}$ with respect to $(\mathcal{Q}_A,C_{A,j})$,
using edges of $H_{A,j}$ (after a suitable orientation of these edges).
Consider any $J^*_{A, {\rm dir}} \in \mathcal{J}^*_{A,j,i}$.%
   \COMMENT{Daniela: changed $\mathcal{J}^*_{A,j}$ to $\mathcal{J}^*_{A,j,i}$}
Lemma~\ref{sysdecom}(a$_3$) guarantees that
$V(J^*_{A, {\rm dir}}) \subseteq A_i$, and so $J^*_{A, {\rm dir}}$ is an $A_i$-extension of itself.
Therefore, in order to find a balanced extension of $\mathcal{J}^*_{A,j}$, it is enough to extend each $J^*_{A, {\rm dir}} \in \mathcal{J}^*_{A,j}$
into a locally balanced path sequence by adding a directed matching which is vertex-disjoint from $J^*_{A, {\rm dir}}$ in such a way that~(BE3)
is satisfied as well.

\begin{lemma}\label{balanceextension}
Suppose that $0<1/m \ll \eps  \ll 1$ and%
\COMMENT{Andy: added and} 
that $m,k \in \mathbb N$ with $k \ge 3$.%
\COMMENT{Need $k \ge 3$ to have a cycle. Note the switch from $K$ to $k$ is intentional here, as will have $k=2K$ in the bip section}
Let $\mathcal{Q} =\{V_1, \dots, V_k\}$ be a $(k,m)$-equipartition of a set $V$ of vertices and let $C = V_1 \dots V_k$ be a directed cycle.
Suppose that there exist a set $\mathcal{M}$ and a graph $H$ on $V$ such that the following conditions hold:
\begin{itemize}
\item[{\rm (i)}] $\mathcal{M}$ can be partitioned into $k$ sets $\mathcal M_1,\dots, \mathcal M_k$%
   \COMMENT{Daniela: reworded}
such that $|\mathcal M_i| \le m/k$ and each $M \in \mathcal M_{i}$ is an ordered directed matching with $e(M)\le \eps m$ and $V(M)\subseteq V_i$
(for all $i\le k$).
\item[{\rm (ii)}] $H[V_{i-1},V_{i+1}]$ is a $2 \eps m$-regular graph for all $i\le k$.
\end{itemize}
Then there exist an orientation $H_{{\rm dir}}$ of $H$ and a balanced extension $\mathcal{BE}$ of $\mathcal{M}$ with respect to
$(\mathcal{Q},C)$ and parameters $( 2 \eps , 3 )$ such that each path sequence in $\mathcal{BE}$ is obtained from some
$M\in \mathcal{M}$ by adding edges of~$H_{{\rm dir}}$.
\end{lemma}

\begin{proof}
Fix $i \le k$ and write $\mathcal{M}_i := \{ M_1, \dots, M_{|\mathcal{M}_i|}\}$.
We orient each edge in $H[V_{i-1},V_{i+1}]$ from $V_{i-1}$ to $V_{i+1}$.
By (ii), $H[V_{i-1},V_{i+1}]$ can be decomposed into $2 \eps m$ perfect matchings. 
Each perfect matching can be split into $1/2\eps $ matchings, each containing at least $ \eps m$ edges.%
    \COMMENT{Not sure why we need $1/2\eps $ instead of $1/\eps $ here. Is it because of the rounding? (Allan: Yes, it is due to rounding.)}
Recall from (i) that $|\mathcal{M}_i| \le m/k$ and $e(M_j ) \le \eps m$ for all $M_j  \in \mathcal{M}_i$.
Hence, $H[V_{i-1},V_{i+1}]$ contains $|\mathcal{M}_i|$ edge-disjoint matchings $M'_1, \dots, M'_{|\mathcal{M}_i|}$ with $e(M'_j) = e(M_j)$ for all $j \le |\mathcal{M}_i|$.
Define $PS_j:=M_j \cup M_j'$.
Note that $PS_j$ is locally balanced with respect to $C$.
Also, $PS_j$ is a $V_i$-extension of $M_j$ (as $M_j \in \mathcal{M}_i$ and so $V(M_j) \subseteq V_i$ by~(i)).
Moreover,
\begin{align}\label{eq:new}
	|V(PS_j) \cap V_{i'}|  = \begin{cases}
		|V(M_j)| =2 e(M_j ) \le 2 \eps m & \textrm{if $i' = i$,}\\
		e(M_j ) \le \eps m & \text{if $i' = i+1$ or $i' = i-1$,}\\
		0 & \text{otherwise.}
	\end{cases}
\end{align}
For each $i\le k$, set $\mathcal{PS}_i := \{ PS_1, \dots, PS_{|\mathcal{M}_i|}\}$.
Therefore, $\mathcal{BE}: = \bigcup_{i \le k} \mathcal{PS}_i$ is a balanced extension of $\mathcal{M}$ with respect to $(\mathcal{Q},C)$ and parameters $(2\eps, 3)$.
Indeed, (BE3) follows from (\ref{eq:new}). As remarked after the definition of a balanced extension, this also implies the
`moreover part' of (BE2).%
    \COMMENT{Daniela: last 2 sentences are new. Previously had "For (BE2), note that the number of $V_i$-extensions is $|\mathcal{M}_i| \le m/k$."}
Hence the lemma follows (by orienting the remaining edges of $H$ arbitrarily).
\end{proof}


\section{Constructing Hamilton cycles via balanced extensions}  \label{sec:extendmerge}

Recall that a cyclic system can be viewed as a blow-up of a Hamilton cycle.%
   \COMMENT{Deryk: new sentence}
Given a cyclic system $(G,\mathcal{Q}, C)$ and a balanced extension~$\mathcal{BE}$ of a set $\mathcal{M}$
of ordered directed matchings,
our aim is to extend each path sequence in $\mathcal{BE}$ into a Hamilton cycle using edges of~$G$.
Moreover, each Hamilton cycle has to be consistent with a distinct $M \in \mathcal{M}$. This is achieved by the
following lemma, which is the key step in proving Lemma~\ref{almostthm}.%
   \COMMENT{Deryk: new sentence}

\begin{lemma}\label{merging}
Suppose that $0<1/m \ll \eps_0, \eps' , 1/k \ll 1/\ell , \rho  \leq 1$, that $0 \le \mu , \rho \ll 1$\COMMENT{It is intentional that I have stated two conditions with $\rho$ in. Thought
it was the easiest way to write $\eps_0, \eps' , 1/k \ll 1/\ell , \rho $, and $1/\ell \leq 1$ and
$\rho \ll 1$.}
 and that $m,k, \ell, q\in \mathbb{N}$ with $q \le (1- \mu - \rho) m$.
Let $(G,\mathcal{Q}, C)$ be a $(k,m, \mu, \eps')$-cyclic system and let $\mathcal{M} = \{M_1, \dots, M_q\}$
be a set of $q$ ordered directed matchings.%
	\COMMENT{Allan: $M_1, \dots, M_1$ are no longer edge-disjoint. Reason:
In two cliques case, $J^*_A = M_1$. This has lots of knock on effect on the rest....
The next two sentences also changed.}
Suppose that $\mathcal{BE} = \{ PS_1, \dots, PS_q \}$ is a balanced extension of $\mathcal{M}$ with respect to $(\mathcal{Q},C)$ and
parameters $(\eps_0, \ell)$ such that for each $s \le q$, $M_s \subseteq PS_s$.
Then there exist $q$ Hamilton cycles $C_1, \dots, C_q$ in $G + \mathcal{BE}$ such that for all $s \le q$, $C_s$ contains $PS_s$ and is consistent with $M_s$, and such that $C_1 - PS_1, \dots, C_q - PS_q$ are edge-disjoint subgraphs of $G$.
\end{lemma}

Lemma~\ref{merging} will be used both in the two cliques case (i.e.~to prove Lemma~\ref{almostthm}) and in the bipartite case (i.e.~to prove Lemma~\ref{almostthmbip}).

We now give an outline of the proof of Lemma~\ref{merging}, where for simplicity we assume that
the path sequences in the balanced extension~$\mathcal{BE}$ are edge-disjoint from each other.%
	\COMMENT{Deryk: reworded}
Our first step is to remove a sparse subdigraph $H$ from $G$ (see Lemma~\ref{slicelemma}), and set $G':= G - H$.
Next we extend each path sequence in~$\mathcal{BE}$ into a (directed) $1$-factor using edges of~$G'$
such that all these $1$-factors are edge-disjoint from each other (see Lemma~\ref{extend1-factor}). 
Finally, we use edges of $H$ to transform the $1$-factors into Hamilton cycles (see Lemma~\ref{mergecycles2}).

The following lemma enables us to find a suitable sparse subdigraph $H$ of $G$.
Recall that $(\eps, d, d^*, c)$-superregularity was defined in Section~\ref{sec:reg}.%
   \COMMENT{Daniela: statement and proof of Lemma~\ref{slicelemma} are slightly different now, since there
was a problem with the application of this lemma in the proof of Lemma~\ref{merging} - CHECK Lemma~\ref{slicelemma} and
the proof of Lemma~\ref{merging}!}

\begin{lemma} \label{slicelemma}
Suppose that $0<1/m\ll \eps' \ll  \gamma \ll \eps\ll 1$ and $0 \le \mu \ll \eps$.
Let $G$ be a bipartite graph with vertex classes $U$ and $W$ of size $m$ such that 
$d(v) = ( 1- \mu \pm \eps')m$ for all $v \in V(G)$.
Then there is a spanning subgraph $H$ of $G$ which satisfies the following properties: 
\begin{itemize}
\item[{\rm (i)}]  $H$ is $(\eps, 2\gamma , \gamma , 3 \gamma )$-superregular.
\item[{\rm (ii)}] Let $G':=G-H$. Then $d_{G'} (v) = (1 - \mu   \pm 4\gamma)m$ for all $v \in V(G)$.
\end{itemize}
\end{lemma}

\begin{proof}
Note that $\delta(G) \ge (1- \mu - \eps') m \ge (1- \eps^3) m$ as $\eps', \mu \ll \eps \ll 1$. 
Thus, whenever $A \subseteq U$ and $B \subseteq W$ are sets of size at least $\eps m,$\COMMENT{B\'ela: added comma} then
\begin{align}\label{exp1a}
	e_G(A,B) \ge (|B| - \eps^3 m)|A| \ge (1 - \eps^2)|A||B|.
\end{align}
Let $H$ be a random subgraph of $G$ which is obtained by including each edge of $G$ with probability $2\gamma $.
\eqref{exp1a} implies that whenever $A \subseteq U$ and $B \subseteq W$ are sets of size at least $\eps m$ then
\begin{align}\label{exp1}
2\gamma (1- \eps^2)|A||B| \leq \mathbb E (e_{H} (A,B)) \leq 2\gamma |A||B|.
\end{align}
Further, for all $u,u' \in V(H)$,
\begin{align}\label{exp2}
 \mathbb E (|N_{H} (u) \cap N_{H} (u')| ) \leq 4 \gamma ^2 m
\end{align}
and 
\begin{align}\label{exp3}
3\gamma m/2 \leq \mathbb E (\delta (H)),  \mathbb E (\Delta (H) )\leq 2 \gamma m .
\end{align}
Thus, \eqref{exp1}--\eqref{exp3} together with Proposition~\ref{chernoff} imply that, with high probability, $H$
is an $(\eps, 2\gamma , \gamma , 3 \gamma )$-superregular pair.
Since $\Delta (H) \leq 3 \gamma m$ by~(Reg3) and $\eps'\ll \gamma$, $G'$ satisfies~(ii).
\end{proof}

\subsection{Transforming a balanced extension into $1$-factors} \label{sec:extend}

The next lemma will be used to extend each locally balanced path sequence $PS$ belonging to a balanced extension $\mathcal{BE}$ into a
(directed) $1$-factor using edges from $G'$.
We will select the edges from $G'$ in such a way that (apart from the path sequences)%
   \COMMENT{Deryk: added brackets}
the $1$-factors obtained are edge-disjoint.

\begin{lemma}\label{extend1-factor}
Suppose that $0<1/m \ll  1/k \le \eps \ll\rho , 1/\ell \leq 1$, that $\rho \ll 1$,
that $0 \le \mu \le 1/4$ and that $q, m,k,\ell \in \mathbb{N}$
with $q \le (1- \mu - \rho) m$.
Let $(G,\mathcal{Q}, C)$ be a $(k,m, \mu, \eps)$-cyclic system, where $C= V_1 \dots V_k$.
Suppose that there exists a set $\mathcal{PS}$ of $q$ path sequences $PS_1,\dots, PS_q$ satisfying the following conditions:
\COMMENT{Allan: had $q$ edge-disjoint path sequences, also changed the last sentence in the statement.}
\begin{itemize}
	\item[\rm (i)] Each $PS_s \in \mathcal{PS}$ is  locally balanced with respect to $C$.
	\item[\rm (ii)] $|V(PS_s) \cap V_i |\le \eps m$ for all $i \le k $ and $s \le q$.
Moreover, for each $i \le k$, there are at most $\ell m/k$ $PS_s$ such that $V(PS_s) \cap V_i \ne \emptyset$.
\end{itemize}
Then there exist $q$ directed $1$-factors $F_1, \dots, F_q$ in $G+ \mathcal{PS}$ such that for all $s \le q$
$PS_s \subseteq F_s$ and $F_1 - PS_1, \dots, F_q - PS_q$ are edge-disjoint subgraphs of $G$.
\end{lemma}

\proof
By changing the values of $\rho$, $\mu$ and $\eps$ slightly,%
   \COMMENT{We do need to make $\eps$ slightly bigger for this as well (in order to ensure that $(G,\mathcal{Q}, C)$
is still a $(k,m, \mu, \eps)$-cyclic system)}
we may assume that $\rho m, \mu m \in \mathbb{N}$.
For each $s \le q$ and each $i \le k$, let $V_i^{s,-}$ (or $V_i^{s,+}$) be the set of vertices in $V_i$ with indegree (or outdegree) one in $PS_s$.
Since each $PS_s$ is locally balanced with respect to $C$, $|V_i^{s,+}| = |V_{i+1}^{s,-}| \le \eps m $ for all $s \le q$ and all $i \le k$
(where the inequality follows from~(ii)).
To prove the lemma, it suffices to show that for each $i \le k$, there exist edge-disjoint directed matchings $M^1_i, \dots, M^q_i$,
so that each $M^s_i$ is a perfect matching in $G[V_i \setminus V_i^{s,+},V_{i+1} \setminus V_{i+1}^{s,-}]$.
The lemma then follows by setting $F_s := PS_s + \sum_{i \le k} M^s_i$ for each $s \le q$.

Fix any $i \le k$.
Without loss of generality (by relabelling the $PS_s$ if necessary) we may assume that there exists an integer $s_0$ such that
$V_i^{s,+} \ne \emptyset$ for all $s \le s_0$ and $V_i^{s,+} = \emptyset $ for all $s_0 < s \le q$.
By~(ii), $s_0 \le \ell m/k$.
Suppose that for some $s$ with $1 \le s \le s_0$ we have already found 
our desired matchings $M^1_i, \dots, M^{s-1}_i$ in $G[V_i,V_{i+1}]$. 
Let 
$$
V'_i : = V_i \setminus V_i^{s,+}, \ \ \ V'_{i+1} : = V_{i+1} \setminus V_{i+1}^{s,-} \ \ \ \mbox{and} \ \ \
G_s := G[V_i',V_{i+1}'] - \sum_{s' < s} M^{s'}_{i}.
$$
Note that each $v \in V_i'$ satisfies
\begin{align} \nonumber
	d^+_{G_s}(v) \ge d^+_{G}(v,V_{i+1}) - ( |V_{i+1}^{s,-}|+s_0) 
 \ge (1- \mu - (2\eps + \ell/k) ) m 
\ge (1- \mu - \sqrt{\eps}) m  
\end{align}
by~(Sys2) and the fact that $ 1/k \le \eps \ll 1/{\ell}$.
Similarly, each $v \in V_{i+1}'$ satisfies $d^-_{G_s}(v) \ge (1- \mu - \sqrt{\eps} ) m$.
Thus $G_s$ contains a perfect matching $M_i^{s}$ (this follows, for example, from Hall's theorem).
So we can find edge-disjoint matchings $M^1_i, \dots, M^{s_0}_i$ in $G[V_i,V_{i+1}]$.

Let $G'$ be the subdigraph of $G[V_i,V_{i+1}]$ obtained by removing all the edges in $M^{1}_{i}, \dots, M^{s_0}_{i}$.
Since $V_i^{s,+} = \emptyset $ for all $s_0 < s \le q$ (and thus also $V_{i+1}^{s,-} = \emptyset$ for all such~$s$),%
	\COMMENT{Allan: had $V_{i}^{s,-}$ before.}
in order to prove the lemma it suffices to find $q-s_0$ edge-disjoint perfect matchings in $G'$.
Each $v \in V_i$ satisfies 
\begin{align*}
d^+_{G'}(v) = d^+_{G}(v,V_{i+1}) \pm  s_0 = d^+_{G} (v,V_{i+1}) \pm \ell m/k = (1- \mu \pm \sqrt{\eps}) m 
\end{align*}
by~(Sys2) and the fact that  $ 1/k \le \eps \ll 1/{\ell}$.
Similarly, each $v \in V_{i+1}$ satisfies $d^-_{G'}(v) = (1- \mu \pm \sqrt{\eps}) m$.
Set $\rho' : = \rho + s_0/m$. 
Note that $\rho' m \in \mathbb{N}$ and $\rho \le \rho' \le \rho + \ell/k \le 2\rho$ as $1/k \ll 1/\ell, \rho$.
Hence, $\eps \ll \rho' \ll 1$. Thus we can apply Lemma~\ref{regularsub} with $G', \rho', \sqrt{\eps}$
playing the roles of $\Gamma,\rho,\eps$ to obtain $(1-\mu - \rho') m $ edge-disjoint perfect matchings in~$G'$.
Since $(1-\mu - \rho') m  =  (1-\mu - \rho) m - s_0 \ge  q- s_0$, there exists $q-s_0$ edge-disjoint
perfect matchings $M_i^{s_0+1}, \dots, M_i^{q}$ in $G'$. This completes the proof of the lemma.
\endproof

\subsection{Merging cycles to obtain Hamilton cycles} \label{sec:merge}

Recall that we have removed a sparse subdigraph $H$ from $G$ and that $G'=G-H$.
Our final step in the proof of Lemma~\ref{merging} is to merge the cycles from each of the $1$-factors $F_s$ returned by
Lemma~\ref{extend1-factor} to obtain edge-disjoint (directed) Hamilton cycles. 
We will apply Lemma~\ref{mergecycles} to merge the cycles of each $F_s$, using the edges in $H$.
However, the Hamilton cycles obtained in this way might not be consistent with the matching $M_s \in \mathcal{M}$
that lies in $PS_s$. Lemma~\ref{ordercycle} is designed to deal with this issue.

Lemma~\ref{mergecycles} was proved in~\cite{Kelly} and was first used to construct approximate Hamilton decompositions in~\cite{OS}.
Roughly speaking, it asserts the following:
suppose that we have a $1$-factor $F$ where most of the edges wind around a cycle $C=V_1\dots V_k$.
Suppose also that we have a digraph $H$ which winds around~$C$. (More precisely, $H$ is the
union of superregular pairs $H[V_i,V_{i+1}]$.)
Then we can transform $F$ into a Hamilton cycle $C'$ by using a few edges of $H$.
The crucial point is that when applying this lemma, the edges in $C'-F$ can be taken from a small number of the
superregular pairs $H[V_i,V_{i+1}]$ (i.e.~the set $J$ in Lemma~\ref{mergecycles} will be very small compared to~$k$).
In this way, we can transform many $1$-factors $F$ into edge-disjoint Hamilton cycles without using any of the pairs $H[V_i,V_{i+1}]$ too often.
This in turn means that we will be able to transform all of our $1$-factors into edge-disjoint Hamilton cycles
by using the edges of a single sparse graph~$H$.

\begin{lemma}\label{mergecycles}
Suppose that $0<1/m\ll d'\ll \eps\ll d\ll \zeta ,1/t\le 1/2$.
Let $V_1,\dots,V_k$ be pairwise disjoint clusters, each of size $m$, and let $C=V_1\dots V_k$ be a directed cycle on these clusters.
Let $H$ be a digraph on $V_1\cup \dots\cup V_k$ and let $J\subseteq E(C)$. For each edge $V_iV_{i+1}\in J$, let $V^1_i\subseteq V_i$
and $V^2_{i+1}\subseteq V_{i+1}$ be such that $|V^1_i|=|V^2_{i+1}|\ge m/100$ and 
 such that $H[V^1_i,V^2_{i+1}]$ is $(\eps,d',\zeta d',td'/d)$-superregular.
Suppose that $F$ is a $1$-regular digraph with $V_1\cup \dots \cup V_k\subseteq V(F)$ such that the following properties hold:
\begin{itemize}
\item[\rm{(i)}]For each edge $V_iV_{i+1}\in J$ the digraph $F[V^1_i,V^2_{i+1}]$ is a perfect matching.
\item[\rm{(ii)}] For each cycle $D$ in $F$ there is some edge $V_iV_{i+1}\in J$ such that $D$ contains a vertex
in $V^1_i$.
\item[\rm{(iii)}] Whenever $V_iV_{i+1}, V_jV_{j+1}\in J$ are such that $J$ avoids all edges in the segment $V_{i+1}CV_j$ of
$C$ from $V_{i+1}$ to $V_j$, then $F$ contains a path $P_{ij}$ joining some vertex $u_{i+1}\in V^2_{i+1}$ to some
vertex $u'_j\in V^1_j$ such that $P_{ij}$ winds around~$C$.
\end{itemize}
Then we can obtain a directed%
    \COMMENT{Daniela: added directed}
cycle on $V(F)$ from $F$ by replacing $F[V^1_i,V^2_{i+1}]$ with a suitable perfect matching
in $H[V^1_i,V^2_{i+1}]$ for each edge $V_iV_{i+1}\in J$.
\end{lemma}

\begin{lemma}\label{ordercycle}
Suppose that $0<1/m\ll \gamma \ll d' \ll \eps\ll d\ll \zeta ,1/t\le 1/2$.
Let $V_1,\dots,V_k$ be pairwise disjoint clusters, each of size $m$, and let $C=V_1\dots V_k$ be a directed cycle on these clusters.
Let $1\le i\le k$ be fixed and let $V^1_i\subseteq V_i$ and $V^2_{i+1}\subseteq V_{i+1}$ be such that $|V^1_i|=|V^2_{i+1}|\ge m/100$. Suppose that 
$H=H[V^1_i,V^2_{i+1}]$ is an $(\eps,d',\zeta d',td'/d)$-superregular bipartite digraph.
Let $X= \{x_1, \dots, x_p\} \subseteq V_{i}^1$ with $|X| \le \gamma m$.
Suppose that $C'$ is a directed cycle with $V_1\cup \dots \cup V_k\subseteq V(C')$ such that $C'[V^1_i,V^2_{i+1}]$ is a perfect matching.
Then we can obtain a directed cycle on $V(C')$ from $C'$ that visits the vertices $x_1, \dots, x_p$ in order by
replacing $C'[V^1_i,V^2_{i+1}]$ with a suitable perfect matching in $H[V_i^1,V_{i+1}^2]$.
\end{lemma}
\proof
Pick $\nu$ and $\tau$ such that $ \gamma \ll \nu \ll \tau\ll d'$.
For every $u \in V^1_{i}$, starting at $u$ we move along 
the cycle $C'$ (but in the opposite direction to the orientation of the edges) and let $f(u)$ be the first vertex on $C'$ in
$V^2_{i+1}$. (Note that $f(u)$ exists since $C'[V^1_i,V^2_{i+1}]$ is a perfect matching.
Moreover,
$f(u) \not = f(v)$ if $u \not = v$.)
Define an auxiliary digraph $A$ on $V^2_{i+1}$ such that $N^+_A(f(u)):=N^+_{H}(u)$.
So $A$ is obtained by identifying each pair $(u,f(u))$ into one vertex with an edge from $(u,f(u))$ to
$(v,f(v))$ if $H$ has an edge from $u$ to $f(v)$. So Lemma~\ref{regtoexpander} applied
with $d'$, $d/t$ playing the roles of $d$, $\mu$ implies
that $A$ is a robust $(\nu,\tau)$-outexpander.
Moreover, $\delta^+(A), \delta^-(A)\ge\zeta d'|V^2_{i+1}|=\zeta d'|A|$ by (Reg4).%
    \COMMENT{Actually $A$ might have loops. But after deleting them we still have a robust $(\nu,\tau)$-outexpander
with $\delta^0(A)\ge \zeta d'|A|-1$. So it's maybe better to gloss over it...}
Thus Theorem~\ref{expanderthm} implies that $A$ has a Hamilton cycle visiting $f(x_1), \dots, f(x_p)$ in order, which clearly corresponds to a perfect matching $M$ in~$H$ with the desired property.
\endproof

The above proof idea is actually quite similar to that for Lemma~\ref{mergecycles} itself.
We now apply Lemmas~\ref{mergecycles} and~\ref{ordercycle} to each $1$-factor $F_s$ given by Lemma~\ref{extend1-factor}
and obtain edge-disjoint Hamilton cycles that are consistent with the~$M_s$.

\begin{lemma}\label{mergecycles2}
Suppose that $0<1/m \ll  \eps_0, 1/k \ll \gamma  \ll \eps   \ll 1$, that $\gamma \ll 1/ \ell \le 1$ and that $q, m,k,\ell \in \mathbb{N}$.%
      \COMMENT{There is actually no condition on $q$ other than that fact that $q \le \ell m$ which is due to balanced extension.}
Let $\mathcal{Q}=\{V_1,\dots,V_k\}$ be a $(k,m)$-equipartition of a vertex set $V$ and let $C=V_1\dots V_k$ be a directed cycle.
Let $\mathcal{M} = \{M_1, \dots, M_q\}$ be a set of ordered directed matchings.
Suppose that $\mathcal{BE}=\{PS_1,\dots,PS_q\}$ is a balanced extension of $\mathcal{M}$ with  respect
to $(\mathcal{Q},C)$ and parameters $(\eps_0, \ell)$.%
    \COMMENT{Daniela: deleted "such that $M_s \subseteq PS_s$ for all $s \le q$" since it follows from (BE2)}
Furthermore, suppose that there exist $1$-regular digraphs $F_1, \dots, F_q $ on $V$ such that for each $s \le q$, $PS_s \subseteq F_s$ and such that
$F_s - PS_s$ winds around~$C$. 
Let $H$ be a digraph on $V$ which is edge-disjoint from each of $F_1 - PS_1, \dots, F_q - PS_q$%
		\COMMENT{Allan: We do not need $F_1-PS_1, \dots, F_s-PS_s$ are edge-disjoint.}
and such that $H[V_{i},V_{i+1}]$ is
$(\eps,2\gamma ,\gamma   , 3 \gamma)$-superregular for all $i\le k$. 
Then there exist $q$ Hamilton cycles $C_1, \dots, C_q$  in
$F_1+ \dots+F_q+H$ such that $C_s$ contains $PS_s$ and is consistent with $M_s$ for all $s \le q$ and such that
$C_1 - F_1, \dots, C_q - F_q$ are edge-disjoint subgraphs of $H$.%
	\COMMENT{Allan: I do mean $C_1 - F_1, \dots, C_q - F_q$ are pairwise edge-disjoint.}
\end{lemma}
\begin{proof}
Recall from (BE2) that for each $s \le q$ there is some $i_s\le k$ such that $PS_s$ is a $V_{i_s}$-extension of $M_s$.
In particular, $M_s\subseteq PS_s$. Let $I_s$ be the set consisting of all $i \le k$ such that $V_i \cap V(PS_s) \ne \emptyset$.
Since $\mathcal{BE}$ is a balanced extension with parameters $(\eps_0, \ell)$, (BE3) implies that%
    \COMMENT{Deryk: deleted $|\{s : i_s = i\}| \le \ell m/k$ in the display below since it is not used}
for every $i \le k$ we have
\begin{align} \label{sbound}
|\{s : i \in I_s\}| \le \ell m/k.
\end{align}
For each $s\le q$ in turn, we are going to show that there exist Hamilton cycles $C_{1}, \dots, C_{s}$ in $F_1 + \dots + F_s+H$ such that 
\begin{itemize}
	\item[\rm(a$_s$)] $PS_{s'}\subseteq C_{s'}$ and $C_{s'}$ is consistent with $M_{s'}$ for all $s' \le s$,%
   \COMMENT{Daniela: added $PS_{s'}\subseteq C_{s'}$}
	\item[\rm(b$_s$)] $E(C_{s'} - F_{s'}) \subseteq \bigcup_{i \in I_{s'}} E(H[V_i,V_{i+1}])  $ for all $s' \le s$,
	\item[\rm(c$_s$)] $C_1 - F_1, \dots, C_{s}-F_s$ are pairwise edge-disjoint.%
	\COMMENT{Allan: had $C_1, \dots, C_s$}
\end{itemize}
So suppose that for some $s$ with $1\le s\le q$ we have already constructed $C_1, \dots, C_{s-1}$. We now construct $C_s$ as follows. 
Let $H_s:=H - \sum_{s' < s}(C_{s'} -F_{s'})$.%
	\COMMENT{Allan: had $H_s:=H-(C_1+\dots+C_{s-1})$}
Define a new constant $d$ such that $\eps \ll d \ll 1$.

Our first task is to apply Lemma~\ref{mergecycles} to $F_s$ to merge all the cycles in $F_s$ into a Hamilton cycle using only edges of $H_s$.
For each $i \in I_s$, let $V_i^{-}$  be the set of vertices in $V_i$ with indegree one in $PS_s$
and let $V_i^{+}$  be the set of vertices in $V_i$ with outdegree one in $PS_s$.
Set $V_i^1 := V_i \setminus V_i^{+}$ and $V_{i+1}^2 := V_{i+1} \setminus V_{i+1}^{-}$.
Since $PS_s$ is locally balanced, $|V_i^{+}| = |V_{i+1}^{-}| \le \eps_0 m $ for all $i \in I_s$ (where the inequality holds by~(BE3)).
By~(b$_{s-1}$) and \eqref{sbound}, $H_s[V_i,V_{i+1}]$ is obtained from $H[V_i,V_{i+1}]$ by removing at most
$|\{s' <s : i \in I_{s'}\}| \le \ell m /k \le \eps^2\gamma m$ edges from each vertex (as $1/k \ll \eps,\gamma, 1/\ell$).
So by Proposition~\ref{superslice5}, $H_s [V_i, V_{i+1}]$ is still $(2\eps,2\gamma ,\gamma/2 , 3 \gamma)$-superregular for each $i \in I_s$.
Recall that $|V_i \setminus V_i^1| = |V_{i+1} \setminus V_{i+1}^2| \le \eps_0 m$.
Hence $H_s [V_i^1, V_{i+1}^2]$ is $(4\eps,2\gamma ,\gamma/4 , 6 \gamma)$-superregular by Proposition~\ref{superslice6}
and thus also $(4\eps,2\gamma ,\gamma/4 , 4\gamma/d)$-superregular.

Let $E_s := \{V_iV_{i+1} : i \in I_s \}$.
Our aim is to apply Lemma~\ref{mergecycles} with $F_s$, $E_s$, $H_s$, $4 \eps$, $2\gamma$, $2$, $1/8$
playing%
    \COMMENT{Need $1/8$ instead of $1/4$ since the density is $2\gamma$}
the roles of $F$, $J$, $H$, $\eps$, $d'$, $t$, $\zeta$.
Our assumption that $F_s- PS_s$ winds around $C$ implies that for each $i\in I_s$, $F_s[V_i^1,V_{i+1}^2]$ is a perfect matching.
So Lemma~\ref{mergecycles}(i) holds. Note that every final vertex of a nontrivial%
    \COMMENT{Daniela: added nontrivial}
path in $PS_s$ must lie in $\bigcup_{i \in I_s} V_i^1$,
implying Lemma~\ref{mergecycles}(ii).%
   \COMMENT{This implies Lemma~\ref{mergecycles}(ii) for cycles $D$ of $F$ which contain some nontrivial path in $PS_s$. The other cycles of $F$
must wind around $C$ and so satisfy Lemma~\ref{mergecycles}(ii) as well.}
Finally, recall that $|V_i^1|, |V_{i+1}^2| \ge (1- \eps_0) m$ for all $i \in I_s$.
Together with our assumption that $F_s - PS_s$ winds around $C$, this easily implies Lemma~\ref{mergecycles}(iii).
So we can apply Lemma~\ref{mergecycles} to obtain a Hamilton cycle $C'_s$ which is constructed from $F_s$ by replacing
$F_s[V^1_i,V^2_{i+1}]$ with a suitable perfect matching in $H_s[V^1_i,V^2_{i+1}]$ for each $i\in I_s$. In particular,
$PS_s\subseteq C'_s$.

Let $H'_s:=H_s - (C'_{s} -F_{s})$.%
	\COMMENT{Allan: had $H'_s:=H_s-C'_{s}$}
Recall that $M_s$ is an ordered directed matching, say $M_s = \{ e_1, \dots, e_r\}$, and that $PS_s$ is a $V_{i_s}$-extension of $M_s$.
For each $j \le r$, let $P_j$ be the path in $PS_s$ containing $e_j$
and let $x_j$ denote the final vertex of $P_j$. Hence $x_1, \dots, x_r$ are distinct and lie in $V_{i_s}^1$.
Together with (BE3) this implies that $r \le \eps_0 m$.
Note that $H'_s [ V_{i_s}^1,V_{i_s+1}^2]$ is obtained from $H_s[ V_{i_s}^1,V_{i_s+1}^2 ]$ by removing a perfect matching, namely $C'_{s}[V_{i_s}^1, V_{i_s+1}^2]$.%
	\COMMENT{Allan: added $C'_{s}[V_{i_s}^1, V_{i_s+1}^2]$ since Lemma~\ref{ordercycle} needs $C'_{s}[V_{i_s}^1, V_{i_s+1}^2]$ is a perfect matching.}
So by Proposition~\ref{superslice5}, $H_s' [V_{i_s}^1, V_{i_s+1}^2]$ is still $(8\eps,2\gamma ,\gamma/8 , 4\gamma/d)$-superregular.
Apply Lemma~\ref{ordercycle} with
$C'_s$, $i_s$, $H_s'[V^1_{i_s}, V^2_{i_s+1}]$, $\eps_0$, $8\eps$, $2\gamma$, $2$, $1/16$ playing the roles of
$C'$, $i$, $H$, $\gamma$, $\eps$, $d'$, $t$, $\zeta$ to obtain a Hamilton cycle $C_s$ which visits $x_1,\dots, x_r$ in this order
and is constructed from $C'_s$ by replacing the perfect matching $C'_s[V_{i_s}^1,V_{i_s+1}^2]$%
    \COMMENT{Daniela: previously had $C'[V_{i_s},V_{i_s+1}]$}
with a suitable perfect matching in $H'_s[V_{i_s}^1,V_{i_s+1}^2]$.%
    \COMMENT{Daniela: previously had $H'_s[V_{i_s},V_{i_s+1}]$}
In particular, $PS_s\subseteq C_s$.

Note that $E(C_s - F_s) \subseteq  \bigcup_{i \in I_{s}} E(H_s[V_i,V_{i+1}])$, so (b$_s$) and (c$_s$) hold.
Since $PS_s \subseteq C_s$ and $x_j$ is the final vertex of $P_j$ and since $e_j\in E(P_j)$, it follows that $C_s$ visits
the edges $e_1, \dots,e_r$ in order. So $C_s$ is consistent with $M_s$, implying (a$_s$). 
\end{proof}

\removelastskip\penalty55\medskip\noindent{\bf Proof of Lemma~\ref{merging}.}%
	\COMMENT{Allan: the proof changed a bit..}
Let $\mathcal{Q}=\{V_1,\dots,V_k\}$. By relabeling the $V_i$ if necessary, we may assume that $C=V_1\dots V_k$.
Define new constants $\gamma$ and $\eps$ such that $\eps_0, \eps',1/k \ll \gamma \ll \eps,\rho , 1 /{\ell} $ and $\mu \ll \eps\ll 1$.%
    \COMMENT{Daniela: hierarchy and some of the constants below changed because of the modifications to Lemma~\ref{slicelemma}} 
For each $i \le k$ we apply Lemma~\ref{slicelemma} to (the underlying undirected graph of) $G[V_i,V_{i+1}]$ in order to obtain
a spanning subdigraph $H$ of $G$ which satisfies the following properties: 
\begin{itemize}
\item[{\rm (i$'$)}] For each $i \le k$, $H[V_i, V_{i+1}]$ is $(\eps, 2\gamma , \gamma , 3 \gamma )$-superregular.
\item[{\rm (ii$'$)}] Let $G':=G-H$. Then $(G', \mathcal{Q}, C)$ is a $(k,m, \mu,4\gamma)$-cyclic system.
\end{itemize}
Indeed, (ii$'$) follows easily from Lemma~\ref{slicelemma}(ii) and the definition of a $(k,m,\mu,4\gamma)$-cyclic system.
Recall that $\mathcal{BE} = \{PS_1, \dots, PS_q\}$ with $M_s \subseteq PS_s$ for all $s \le q$. Our next aim is to
apply Lemma~\ref{extend1-factor} with $G'$, $\mathcal{BE}$, $4\gamma$ playing the roles of $G$, $\mathcal{PS}$, $\eps$
to obtain $1$-factors $F_s$ extending the $PS_s$.%
    \COMMENT{Deryk: last bit is new}
Note that (BE1) and (BE3) imply that conditions~(i) and~(ii) of Lemma~\ref{extend1-factor} hold.
So we can apply Lemma~\ref{extend1-factor} to obtain $q$ (directed) $1$-factors $F_1, \dots, F_q$ in $G'+\mathcal{BE}$%
   \COMMENT{Daniela: replaced on $V$ by in $G'+\mathcal{BE}$}
such that $PS_s \subseteq F_s$ for all $s \le q$ and $F_1 - PS_1, \dots, F_q-PS_q$ are edge-disjoint subgraphs of $G'$.
Recall from (ii$'$) and (Sys2) that $G'$ (and thus also $F_s - PS_s$) winds around~$C$.
So we can apply Lemma~\ref{mergecycles2} to obtain $q$ Hamilton cycles $C_1, \dots, C_q$ in
$F_1+ \dots+F_q+H$ such that $C_s$ contains $PS_s$ and is consistent with $M_s$ for all $s \le q$, and such that $C_1 - F_1, \dots, C_q-F_q$ are edge-disjoint subgraphs of $H$.
Since $H$ and $G'$%
    \COMMENT{Daniela: had $H$ and $H'$ instead of $H$ and $G'$}
are edge-disjoint, $C_1 - PS_1, \dots, C_q-PS_q$ are edge-disjoint subgraphs of $G$.
\endproof

We can now put everything together to prove the approximate decomposition lemma in the two cliques case.
First we apply Lemma~\ref{sysdecom} to obtain cyclic systems and sparse subgraphs $H_{A,j}$ and $H_{B,j}$.
Then we apply Lemma~\ref{balanceextension} to balance out the exceptional systems into balanced extensions.
Next, we apply Lemma~\ref{merging} to $A$ and $B$ separately to extend the balanced extensions into Hamilton cycles.%
    \COMMENT{Deryk: added blabla}

\removelastskip\penalty55\medskip\noindent{\bf Proof of Lemma~\ref{almostthm}. }
Apply Lemma~\ref{sysdecom} to $G, \mathcal{P}$ and $\mathcal{J}$ to obtain  (for each $1 \le j \le (K-1)/2$)  pairs of 
tuples $(G_{A,j}, \mathcal{Q}_{A}, C_{A,j}, H_{A,j}, \mathcal{J}^*_{A,j})$ and $(G_{B,j}, \mathcal{Q}_{B}, C_{B,j}, H_{B,j}, \mathcal{J}^*_{B,j})$
which satisfy (a$_1$)--(a$_7$).
Fix $j \le (K-1)/2$. Write $\mathcal{J}^*_{A,j} = \{J^*_{A, {\rm dir},1}, \dots, J^*_{A, {\rm dir},q}\}$, where%
    \COMMENT{Deryk: defined $q$ and used it below}
\begin{equation}\label{eq:q}
q:=|\mathcal{J}^*_{A,j}|\le (1-4\mu-3\rho)m
\end{equation}
by~(a$_3$).
We now apply Lemma~\ref{balanceextension} with $\mathcal{J}^*_{A,j},\mathcal{Q}_A,C_{A,j}, H_{A,j}, K, 5K \sqrt{\eps_{0}}$ playing
the roles of $\mathcal{M}, \mathcal{Q}, C, H, k,\eps$ to obtain an orientation $H_{A,j,{\rm dir}}$ of $H_{A,j}$ and
a balanced extension $\mathcal{BE}_j$ of $\mathcal{J}^*_{A,j}$ with respect to $(\mathcal{Q}_A, C_{A,j})$ and parameters $(10K \sqrt{\eps_{0}},3)$.
(Note that (a$_3$) and (a$_5$) imply conditions (i) and (ii) of Lemma~\ref{balanceextension}.)
Write $\mathcal{BE}_j :=  \{PS_{1}, \dots, PS_q \}$ such that $J^*_{A, {\rm dir},s} \subseteq PS_s$ for all $s \le q$.
So (BE1) implies that%
   \COMMENT{Daniela: reworded}
$PS_{1} - J^*_{A, {\rm dir},1}, \dots, PS_q - J^*_{A, {\rm dir}, q}$ are edge-disjoint subgraphs of $H_{A,j,{\rm dir}}$.%
	\COMMENT{Allan:added more explanations}
Since $(G_{A,j, {\rm dir}}, \mathcal{Q}_{A}, C_{A,j})$ is a $(K,m, 4 \mu, 5/K)$-cyclic system by (a$_6$), (\ref{eq:q}) implies that%
    \COMMENT{Deryk: added ref to (\ref{eq:q})}
we can apply Lemma~\ref{merging} as follows:
\smallskip
\begin{center}
  \begin{tabular}{ r | c | c | c | c | c | c | c | c | c | c |c}

& $G_{A,j, {\rm dir}}$ & $\mathcal{Q}_{A}$ & $C_{A,j}$ &  $K$ & $\mathcal{J}^*_{A,j}$ & $q$ & $4 \mu$ & $3\rho$ & $10K \sqrt{\eps_{0}}$ & $5/K$ & $3$ \\ \hline
plays role of & $G$ & $\mathcal{Q}$ & $C$ & $k$ &  $\mathcal{M}$ & $q$ & $\mu$  & $\rho$ &  $\eps_0$ & $\eps'$ &  $\ell$ \\
  \end{tabular}
\end{center}
\smallskip
\noindent
In this way we obtain $q$ directed Hamilton cycles $C'_{A,j,1}, \dots, C'_{A,j,q}$
in $G_{A,j, {\rm dir}} + \mathcal{BE}_j$ such that $C'_{A,j,s}$ contains $PS_{s}$ and is consistent with $J^*_{A, {\rm dir},s}$ for all $s \le q$.
Moreover, $C'_{1} - J^*_{A, {\rm dir},1}, \dots, C'_q - J^*_{A, {\rm dir}, q}$ are edge-disjoint subgraphs of $G_{A,j,{\rm dir}} + H_{A,j,{\rm dir}}$.%
	\COMMENT{Allan: added details here, the rest of the proof remains the same.}
Repeat this process for all $j \le (K-1)/2$.

Write $\mathcal{J} = \{ J_1, \dots, J_{|\mathcal{J}|}\}$.
Recall from (a$_2$) that the $\mathcal{J}^*_{A,1}, \dots, \mathcal{J}^*_{A,(K-1)/2}$ partition $\{J_{A,{\rm dir}}^* :J \in \mathcal{J}\}$.
Therefore, we have obtained $|\mathcal{J}|$ directed Hamilton cycles%
    \COMMENT{Daniela: had "edge-disjoint directed Hamilton cycles" before}
$C'_{A,1}, \dots, C'_{A,|\mathcal{J}|}$ on vertex set $A$.
Moreover, by relabelling the $J_s$ if necessary, we may assume that $C'_{A,s}$ is consistent with $(J_s)^*_{A, {\rm dir}}$ for all $s \le |\mathcal{J}|$.
Furthermore, (a$_4$) implies that the undirected versions of $C'_{A,1} - (J_1)^*_{A, {\rm dir}}, \dots, C'_{A,|\mathcal{J}|} - (J_{|\mathcal{J}|})^*_{A, {\rm dir}}$
are edge-disjoint spanning subgraphs of $G[A]$.

Similarly we obtain directed Hamilton cycles%
    \COMMENT{Daniela: had "edge-disjoint directed Hamilton cycles" before}
$C'_{B,1}, \dots, C'_{B,|\mathcal{J}|}$ on vertex set $B$
so that $(J_s)^*_{B, {\rm dir}} \subseteq C'_{B,s}$ for all $s \le |\mathcal{J}|$.
Let $H_s$ be the undirected graph obtained from $C'_{A,s} + C'_{B,s}-J_s^* +J_s$ by ignoring all the orientations of the edges.
Recall that $J_1, \dots, J_{|\mathcal{J}|}$%
\COMMENT{Andy: replaced subscript $s$ with subscript $|\mathcal{J}|$.}
are edge-disjoint exceptional systems and that they are edge-disjoint from the $C'_{A,s} + C'_{B,s}-J_s^*$
by (EC3).%
    \COMMENT{Deryk: added detail}
So $H_1, \dots, H_{|\mathcal{J}|}$ are edge-disjoint spanning subgraphs of $G$.
Finally, Proposition~\ref{prop:ES2} implies that $H_1, \dots, H_{|\mathcal{J}|}$ are indeed as desired in Lemma~\ref{almostthm}.
\endproof


\section{The bipartite case}
Roughly speaking, the idea in this case is to reduce the problem of finding the desired edge-disjoint Hamilton cycles in~$G$
to that of finding suitable Hamilton cycles in an almost complete balanced bipartite graph.
This is achieved by considering the graphs $J^*_{\rm dir}$, which we define in the next subsection.
The main steps are similar to those in the proof of Lemma~\ref{almostthm}
(in fact, we re-use several of the lemmas, in particular Lemma~\ref{merging}).

We will construct the graphs $J^*_{\rm dir}$, which are based on balanced exceptional systems~$J$, in Section~\ref{sec:J*2bip}.
In Section~\ref{systembip} we describe a decomposition of~$G$ into blown-up Hamilton cycles.
We will construct balanced extensions in Section~\ref{sec:BEbip}
(this is more difficult than in the two cliques case).
Finally, we obtain the desired Hamilton cycles using Lemma~\ref{merging} (in the same way as in the two cliques case).

\subsection{Defining the graphs $J^*_{\rm dir}$ for the bipartite case} \label{sec:J*2bip}

Let $\mathcal{P}$ be a $(K,m,\eps)$-partition of a vertex set $V$ and
let $J$ be a balanced exceptional system with respect to~$\mathcal{P}$ (as defined in Section~\ref{mainbi}).
We construct $J^{*}$ in two steps. 
First we will construct a matching $J^*_{AB}$ on $A \cup B$
and then $J^{*}.$%
	   \COMMENT{Allan: I have rephrased the definition/construction of $J^*$. $J^*$ is still the same as before.
The construction is similar to the two cliques case.  So check whether things
are ok now. Note we no longer have conditions (BES$'$1), (BES$'$2), (BES$^*$1),(BES$^*$2) but they don't appear anywhere here (and only once in paper 2).}
Since each maximal path in $J$ has endpoints in $A \cup B$ and internal vertices in $V_0$ by (BES1), a balanced exceptional system $J$ naturally induces a matching $J^*_{AB}$ on $A \cup B$.
More precisely, if $P_1, \dots ,P_{\ell'}$ are the non-trivial paths in~$J$ and $x_i, y_i$ are the endpoints of $P_i$, then we define $J^*_{AB} := \{x_iy_i : i  \le \ell'\}$. 
Thus $J^*_{AB}$ is a matching by~(BES1) and $e(J^*_{AB}) \le e(J)$.%
    \COMMENT{Daniela: deleted  by~(BES4)}
Moreover, $J^*_{AB}$ and $E(J)$ cover exactly the same vertices in $A$. 
Similarly, they cover exactly the same vertices in $B$. 
So (BES3) implies that $e(J^*_{AB}[A])=e(J^*_{AB}[B])$.
We can write $E(J^*_{AB}[A])=\{x_1x_2, \dots, x_{2s-1}x_{2s}\}$,
$E(J^*_{AB}[B])=\{y_1y_2, \dots, y_{2s-1}y_{2s}\}$ and $E(J^*_{AB}[A,B])=\{x_{2s+1}y_{2s+1}, \dots, x_{s'}y_{s'}\}$, where $x_i \in A$ and $y_i \in B$.
Define $J^*:= \{ x_iy_i : 1 \le i \le s' \}$.
Note that 
\begin{align}
	e(J^*) =  e(J^*_{AB}) \le e(J). \label{BESeq}
\end{align}
All edges of $J^*$ are called \emph{fictive edges}.%
   \COMMENT{Daniela: deleted "Similarly as in the two cliques case, we regard $J^*$ as being edge-disjoint from the original graph~$G$."}
We say that an (undirected) cycle $D$ is \emph{consistent with $J^*$} if $D$ contains $J^*$ and (there is an orientation of $D$ which)
visits the vertices $x_1,y_1,x_2,\dots,y_{s'-1},x_{s'},y_{s'}$ in this order.
The following proposition is proved in~\cite{paper2}.
It is illustrated in Figure~\ref{fig3}.%
	\COMMENT{Allan:New sentence and figure.}

\begin{figure}[tbp]
\centering
\subfloat[$J$]{
\begin{tikzpicture}[scale=0.35]
			\draw  (-3,0) ellipse (1.5  and 3.8 );
			\draw (3,0) ellipse (1.5  and 3.8 );
			\node at (-3,-4.5) {$A$};
			\node at (3,-4.5) {$B$}; 			
			\draw (-3,5) circle(0.5);
			\draw (3,5) circle (0.5);
			\node at (-3,6) {$A_0$};
			\node at (3,6) {$B_0$};
			\fill (-3,5) circle (4pt);
			\fill (3,5) circle (4pt);
			\begin{scope}[start chain]
			\foreach \i in {2.5,1.5,-0.5}
			\fill (-3,\i) circle (4pt);
			\end{scope}
			\begin{scope}[start chain]
			\foreach \i in {2.5,1.5,-0.5}
			\fill (3,\i) circle (4pt);
			\end{scope}

			\begin{scope}[line width=1.5pt]
			\draw (-3,-0.5) to [out=115, in=-115] (-3,5)--(3,-0.5);
			\draw (-3,1.5)--(3,5)--(-3,2.5);
			\draw (3,2.5)--(3,1.5);
			\end{scope}
			
			\begin{scope}[line width=0.6pt]
			\draw (3,2.5)-- (-3,1.5);
			\draw (3,1.5)-- (-3,0.5)--(3,0.5)  -- (-3,-0.5);
			\draw (3,-0.5)--(-3,-1.5) -- (3,-1.5) --(-3,-2.5) -- (3,-2.5) -- (-3,2.5);
			\end{scope}
\end{tikzpicture}
}
\qquad
\subfloat[$J_{AB}^*$]{
\begin{tikzpicture}[scale=0.35]
			\draw  (-3,0) ellipse (1.5  and 3.8 );
			\draw (3,0) ellipse (1.5  and 3.8 );
			\node at (-3,-4.5) {$A$};
			\node at (3,-4.5) {$B$}; 			
			\draw (-3,5) circle(0.5);
			\draw (3,5) circle (0.5);
			\node at (-3,6) {$A_0$};
			\node at (3,6) {$B_0$};
			\fill (-3,5) circle (4pt);
			\fill (3,5) circle (4pt);
			\begin{scope}[start chain]
			\foreach \i in {2.5,1.5,-0.5}
			\fill (-3,\i) circle (4pt);
			\end{scope}
			\begin{scope}[start chain]
			\foreach \i in {2.5,1.5,-0.5}
			\fill (3,\i) circle (4pt);
			\end{scope}

			\begin{scope}[line width=1.5pt]
			\draw (-3,-0.5) --(3,-0.5);
			\draw (-3,1.5)--(-3,2.5);
			\draw (3,2.5)--(3,1.5);
			\end{scope}
			
			\begin{scope}[line width=0.6pt]
			\draw (3,2.5)-- (-3,1.5);
			\draw (3,1.5)-- (-3,0.5)--(3,0.5)  -- (-3,-0.5);
			\draw (3,-0.5)--(-3,-1.5) -- (3,-1.5) --(-3,-2.5) -- (3,-2.5) -- (-3,2.5);
			\end{scope}
\end{tikzpicture}
}
\qquad
\subfloat[$J^*$]{
\begin{tikzpicture}[scale=0.35]
			\draw  (-3,0) ellipse (1.5  and 3.8 );
			\draw (3,0) ellipse (1.5  and 3.8 );
			\node at (-3,-4.5) {$A$};
			\node at (3,-4.5) {$B$}; 			
			\draw (-3,5) circle(0.5);
			\draw (3,5) circle (0.5);
			\node at (-3,6) {$A_0$};
			\node at (3,6) {$B_0$};
			\fill (-3,5) circle (4pt);
			\fill (3,5) circle (4pt);
			\begin{scope}[start chain]
			\foreach \i in {2.5,1.5,-0.5}
			\fill (-3,\i) circle (4pt);
			\end{scope}
			\begin{scope}[start chain]
			\foreach \i in {2.5,1.5,-0.5}
			\fill (3,\i) circle (4pt);
			\end{scope}

			\begin{scope}[line width=1.5pt]
			\draw (-3,-0.5)--(3,-0.5);
			\draw (-3,1.5)--(3,1.5);
			\draw (3,2.5)--(-3,2.5);
			\end{scope}

			\begin{scope}[line width=0.6pt]
			\draw (3,2.5)-- (-3,1.5);
			\draw (3,1.5)-- (-3,0.5)--(3,0.5)  -- (-3,-0.5);
			\draw (3,-0.5)--(-3,-1.5) -- (3,-1.5) --(-3,-2.5) -- (3,-2.5) -- (-3,2.5);
			\end{scope}
\end{tikzpicture}
}
\caption{The thick lines illustrate the edges of $J$, $J_{AB}^*$ and $J^*$, respectively.}

\label{fig3}
\end{figure}
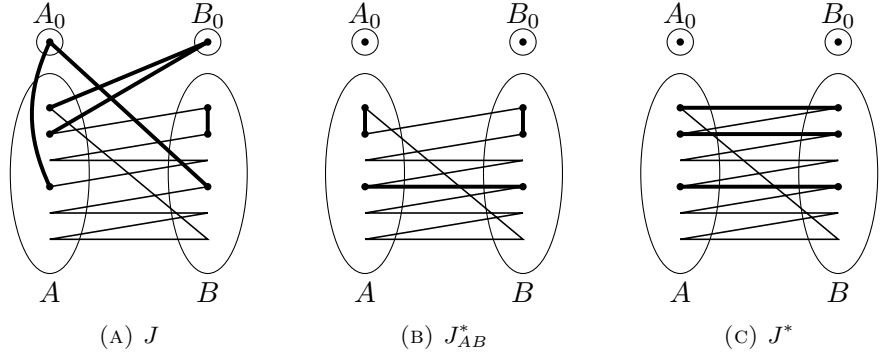

\begin{prop}\label{CES-H}
Let $\cP$ be a $(K,m,\eps)$-partition of a vertex set $V$.
Let $G$ be a graph on $V$ and let $J$ be a balanced exceptional system with respect to~$\cP$. If $J\subseteq G$ and
$D$ is a Hamilton cycle of $G[A\cup B]+ J^*$ which is consistent with $J^{*}$, then $D-J^*+J$ is a Hamilton cycle of $G$.
\end{prop}

Again, we will actually need a directed version, which follows immediately from the above. 
For this, define $J^*_{\rm dir}$ to be the ordered directed matching $\{f_1, \dots, f_{s'}\}$ such that $f_i$ is a
directed edge from $x_i$ to $y_i$ for all $i \le s'$. So $J^*_{\rm dir}$ consists only of $AB$-edges.%
    \COMMENT{Deryk: added new sentence}
Similarly to the undirected case, we say that a directed cycle
$D_{\rm dir}$ is \emph{consistent with $J^*_{\rm dir}$}
if $D_{\rm dir}$ contains $J^*_{\rm dir}$ and visits the edges $f_1,\dots,f_{s'}$ in this order.

\begin{prop}\label{CES-H2}
Let $\cP$ be a $(K,m,\eps)$-partition of a vertex set $V$.
Let $G$ be a graph on $V$ and let $J$ be a balanced exceptional system with respect to~$\cP$ such that $J\subseteq G$.
Suppose that $D_{\rm dir}$ is a directed Hamilton cycle on $A \cup B$ such that $D_{\rm dir}$ is consistent with $J^*_{\rm dir}$.
Furthermore, suppose that $D - J^*  \subseteq G$, where $D$ is the cycle obtained from $D_{\rm dir}$ after ignoring the directions of all edges. 
Then $D-J^*+J$ is a Hamilton cycle of $G$.
\end{prop}

\subsection{Finding systems} \label{systembip}

The following lemma gives a decomposition of an almost complete bipartite graph~$G$ into 
blown-up Hamilton cycles (together with an associated decomposition of exceptional systems).
Its proof is almost the same as that of Lemma~\ref{sysdecom}, so we omit it here.
The only difference is that instead of Walecki's theorem we use a result of Auerbach and Laskar~\cite{auerbach} to decompose the complete 
bipartite graph $K_{K,K}$ into Hamilton cycles, where $K$ is even.%
	\COMMENT{The proof is in an appendix at the end of the file.}
\begin{lemma}\label{sysdecombip}
Suppose that $0<1/n \ll \eps_0  \ll 1/K \ll \rho  \ll 1$ and $0 \leq \mu \ll 1$,
where $n,K \in \mathbb N$ and $K$ is even.
Suppose that $G$ is a graph on $n$ vertices and $\mathcal{P}=\{A_0,A_1,\dots,A_K,B_0,B_1,\dots,B_K\}$ is a $(K, m, \eps _0)$-partition of $V(G)$.
Furthermore, suppose that the following conditions hold:%
\COMMENT{Note that we do not need to assume Lemma~\ref{almostthmbip}(d)}
\begin{itemize}
	\item[{\rm (a)}] $d(v,B_i) = (1 - 4 \mu \pm 4 /K) m $ and $d(w,A_i) = (1 - 4 \mu \pm 4 /K) m $ for all
	$v \in A$, $w \in B$ and $1\leq i \leq K$.
	\item[{\rm (b)}] There is a set $\mathcal J$ which consists of at most $(1/4-\mu - \rho)n$ edge-disjoint  exceptional systems with parameter $\eps_0$ in~$G$.
	\item[{\rm (c)}] $\mathcal J$ has a partition into $K^4$ sets $\mathcal J_{i_1,i_2,i_3,i_4}$ (one for all $1\le \I \le K$) such that each $\mathcal J_{\I}$ consists of precisely $|\mathcal J|/{K^4}$ $\i$-BES with respect to~$\cP$.
\end{itemize}
Then for each $1 \le j \le K/2$, there is a tuple $(G_{j}, \mathcal{Q}, C_{j}, H_{j}, \mathcal{J}_{j})$
such that the following assertions hold, where $\mathcal{Q}:=\{A_1, \dots, A_K,B_1, \dots, B_K \}$:
\begin{itemize}
\item[{\rm (a$_1$)}] Each of $C_{1}, \dots, C_{K/2}$ is a directed Hamilton cycle on $\mathcal{Q}$ 
such that the undirected versions of these cycles form a Hamilton decomposition of the complete bipartite graph
whose vertex classes are $\{A_1, \dots, A_K \}$ and $\{B_1, \dots, B_K \}$.

\item[{\rm (a$_2$)}] $\mathcal{J}_{1}, \dots, \mathcal{J}_{K/2}$ is a partition of $\mathcal{J}$.

\item[{\rm (a$_3$)}] Each $\mathcal{J}_{j}$ has a partition into $K^4$ sets $\mathcal J_{j,\I}$ (one for all $1\le \I\le K$)
such that $\mathcal J_{j,\I}$ consists of $\i$-BES with respect to~$\cP$ and $|\mathcal J_{j,\I}| \le (1- 4\mu - 3\rho) m/K^4$.

\item[{\rm (a$_4$)}] $G_{1},\dots, G_{K/2}, H_{1},\dots,  H_{K/2}$ are edge-disjoint subgraphs of $G[A,B]$.

\item[{\rm (a$_5$)}] $H_{j}[A_{i},B_{i'}]$ is a $(11K + 248/K) \eps_0 m$-regular graph for all $j \le K/2$ and all $i,i' \le K$.%
\COMMENT{We have omitted the floor/ceiling on $(11K + 248/K) \eps_0 m$.}

\item[{\rm (a$_6$)}] For each $j \le K/2$, there exists an orientation $G_{j,{\rm dir}}$ of $G_{j}$ such that $(G_{j,{\rm dir}}, \mathcal{Q}, C_{j})$ is a $(2K,m, 4\mu, 5/K)$-cyclic system.

\end{itemize}
\end{lemma}

\subsection{Constructing balanced extensions} \label{sec:BEbip}
Let $\cP=\{A_0,A_1,\dots,A_K,B_0,B_1,\dots,B_K\}$ be a $(K,m,\eps)$-partition of a vertex set $V$, let $\mathcal{Q} := \{A_1,\dots,A_K,B_1,\dots,B_K \}$
and let $C= A_1B_1A_2B_2 \dots A_KB_K$ be a directed cycle.
Given a set $\mathcal{J}$ of balanced exceptional systems with respect to $\mathcal{P}$, we write $\mathcal{J}^*_{\rm dir}:=\{J^*_{\rm dir}:J\in\mathcal{J}\}$.
So $\mathcal{J}^*_{\rm dir}$ is a set of ordered directed matchings and thus it makes sense to construct a balanced extension of $\mathcal{J}^*_{\rm dir}$
with respect to $(\mathcal{Q},C)$. (Recall that balanced extensions were defined in Section~\ref{system}.)

Now consider any of the tuples $(G_{j}, \mathcal{Q}, C_{j}, H_{j}, \mathcal{J}_{j})$ guaranteed by Lemma~\ref{sysdecombip}.
We will apply the following lemma to find a balanced extension of $(\mathcal{J}_j)^*_{{\rm dir}}$ with respect to $(\mathcal{Q},C_{j})$, using edges of $H_{j}$
(after a suitable orientation of these edges). So the lemma is a bipartite analogue of Lemma~\ref{balanceextension}.
However, the proof is more involved than in the two cliques case.

\begin{lemma}\label{balanceextensionbip}
Suppose that $0<1/n \ll \eps  \ll 1/K \ll 1$,
where $n,K \in \mathbb N$.
Let $\cP=\{A_0,A_1,\dots,A_K,B_0,B_1,\dots,B_K\}$ be a $(K,m,\eps)$-partition of a set $V$ of $n$ vertices.
Let $\mathcal{Q} := \{A_1,\dots,A_K,B_1,\dots,B_K \}$ and let $C:= A_1B_1A_2B_2 \dots A_KB_K$ be a directed cycle. 
Suppose that there exist a set $\mathcal{J}$ of edge-disjoint balanced exceptional systems with respect to $\mathcal{P}$ and parameter $\eps$ and a graph $H$ such that the following conditions hold:
\begin{itemize}
\item[{\rm (i)}]  $\mathcal{J}$ can be partitioned into $K^4$ sets $\mathcal J_{\I}$ (one for all $1\le \I\le K$) such that $\mathcal J_{\I}$
consists of $\i$-BES with respect to~$\cP$ and $|\mathcal J_{\I}| \le m/K^4$.
\item[{\rm (ii)}] For each $v \in A \cup B$ the number of all those $J \in \mathcal{J}$ for which $v$ is incident to an edge in $J$ is at most $2 \eps n $.
\item[{\rm (iii)}] $H[A_{i},B_{i'}]$ is a $(11K+ 248/K)\eps m$-regular graph for all $i,i' \le K$.
\end{itemize}
Then there exist an orientation $H_{{\rm dir}}$ of $H$ and a balanced extension $\mathcal{BE}$ of $\mathcal{J}^*_{\rm dir}$ with respect to $(\mathcal{Q},C)$
and parameters $( 12 \eps K , 12)$ such that each path sequence in $\mathcal{BE}$ is obtained from some $J^*_{\rm dir}\in \mathcal{J}^*_{\rm dir}$
by adding edges of $H_{{\rm dir}}$.
\end{lemma}

The proof proceeds roughly as follows. Consider any $J \in \mathcal{J}_{\I}$.
We extend $J^*_{\rm dir}$ into a locally balanced path sequence in two steps.
For this, recall that $J^*_{\rm dir}$ consists only of edges from $A_{i_1} \cup A_{i_2}$ to $B_{i_3} \cup B_{i_4}$.%
    \COMMENT{Daniela: reworded}
In the first step, we construct a path sequence $PS$ that is an $A_{i_1}$-extension of $J^*_{\rm dir}$ by adding suitable $B_{i_3}A_{i_1}$- and $B_{i_4}A_{i_1}$-edges
from~$H$ to $J^*_{\rm dir}$. In the second step, we locally balance $PS$ in such a way that (BE1)--(BE3) are satisfied.%
    \COMMENT{Deryk: reworded}

\proof
First we decompose $H$ into $H'$ and $H''$ such that $H'[A_i,B_{i'}]$ is a $11 \eps K m$-regular graph for all $i,i' \le K$ and $H'' := H - H'$.
Hence $H''[A_i,B_{i'}]$ is a $248 \eps m/K$-regular graph for all $i,i' \le K$.

Write $\mathcal{J} := \{ J_1, \dots, J_{|\mathcal{J}|} \}$.
For each $s \le |\mathcal{J}|$, we will extend $J^*_{s,{\rm dir}}:=(J_s)^*_{\rm dir}$ into a path sequence $PS_s$ satisfying the following conditions:
\begin{itemize}
	\item[\rm ($\alpha_s$)] Suppose that $J_s \in \mathcal{J}_{\I}$.
Then $PS_s$ is an $A_{i_1}$-extension of $J^*_{s,{\rm dir}}$ consisting of precisely $e(J^*_{s})$ vertex-disjoint directed paths of length two.
	\item[\rm ($\beta_s$)] $V(PS_s) = V(J^*_{s,{\rm dir}}) \cup A'_s$, where $A'_s \subseteq A_{i_1} \setminus V(J^*_{s,{\rm dir}})$ and $|A'_s| = e(J^*_{s})$. 
	\item[\rm ($\gamma_s$)] $PS_s - J^*_{s,{\rm dir}}$ is a matching of size $e(J^*_{s})$ from $B'_s$ to $A'_s$, where $B'_s := V(J^*_{s,{\rm dir}}) \cap (B_{i_3} \cup B_{i_4})$.
	\item[\rm ($\delta_s$)] Let $M_s$ be the set of undirected edges obtained from $PS_s - J^*_{s,{\rm dir}}$ after ignoring all the orientations.
	Then $M_1, \dots, M_s$ are edge-disjoint matchings in $H'$.	
	\item[\rm ($\eps_s$)] $PS_s$ consists only of	edges from $A_{i_1} \cup A_{i_2}$ to $B_{i_3} \cup B_{i_4}$,  and 
from $B_{i_3} \cup B_{i_4}$ to $A_{i_1}$.
\end{itemize}
Note that ($\beta_s$) and ($\gamma_s$) together imply ($\eps_s$).
Suppose that for some $s$ with $1\le s \le |\mathcal{J}|$ we have already constructed $PS_1, \dots, PS_{s-1}$. We will now construct $PS_s$ as follows.
Let $\I$ be such that $J_s \in \mathcal{J}_{\I}$ and let  $H'_s:=H'-(M_1+ \dots +M_{s-1})$.
(BES4) implies that%
\COMMENT{Allan: added stackrel \eqref{BESeq}}
\begin{align}
e( J^*_{s, {\rm dir}} ) = e( J^*_s ) \stackrel{\eqref{BESeq}}{\le} e(J_s) \le \eps n \le 3 \eps K m 
\quad
\text{and}
\quad
|V(J^*_{s,{\rm dir}}) \cap A_{i_1}| \le 3\eps K m .
\label{J^*}
\end{align}
Consider any $s' < s$. Recall from the definition of $J^*_{s',{\rm dir}}$ that $V(J^*_{s',{\rm dir}})$ is the set of all those
vertices in $A\cup B$ which are covered by edges of $J_{s'}$. Together with ($\beta_{s'}$) and ($\gamma_{s'}$) this implies that
a vertex $v \in B$ is covered by~$M_{s'}$ if and only if $v$ is incident to an edge of $J_{s'}$.
Together with (ii) this in turn implies that for all $v \in B$ we have
\begin{eqnarray*}
d_{H'_s}( v , A_{i_1} ) & \ge & d_{H'}( v , A_{i_1} ) - \sum_{s' < s} d_{ M_{s'}}(v) 
\ge 11 K \eps m  - 2\eps n
\ge 11 K \eps m - 5 K \eps m\\
& \stackrel{(\ref{J^*})}{\ge} & |V(J^*_{s,{\rm dir}}) \cap A_{i_1}|  + e(J^*_{s}).
\end{eqnarray*}
Note that $e ( J^*_s )=|V(J^*_{s,{\rm dir}}) \cap (B_{i_3}\cup B_{i_4})|=|B_s'|$. So we can
greedily find a matching $M_s$ of size $e ( J^*_s )$ in $H'_s[A_{i_1}\setminus V(J^*_{s,{\rm dir}}), B_s']$
(which therefore covers all vertices in $B_s'$).
Orient all edges of $M_s$ from $B_s'$ to $A_{i_1}$ and call the resulting directed matching $M_{s, {\rm dir}}$.
Set 
$$
PS_s:= J^*_{s,{\rm dir}} + M_{s, {\rm dir}}.
$$
Note that $PS_s$ consists of precisely $e(J^*_{s})$ directed paths of length two whose final vertices lie in $A_{i_1}$, so ($\alpha_s$)--($\eps_s$) hold by our construction.
This shows that we can obtain path sequences $PS_1, \dots, PS_{|\mathcal{J}|}$ satisfying ($\alpha_s$)--($\eps_s$)
for all $s\le |\mathcal{J}|$.%
   \COMMENT{Daniela: added for all $s\le |\mathcal{J}|$}

The following claim provides us with a `reservoir' of edges which we will use to balance out the edges of each $PS_s$ and thus extend
each $PS_s$ into a path sequence $PS'_s$ which is locally balanced with respect to~$C$.

\medskip

\noindent
{\bf Claim.} 
\emph{$H''$ contains $|\mathcal{J} |$ subgraphs $H''_1, \dots, H''_{|\mathcal{J} |}$ satisfying the following properties
for all $s\le |\mathcal{J} |$ and all $i,i'\le K$:
\begin{itemize}
	\item[\rm (a$_1$)] If $PS_s$ contains an $A_i B_{i'} $-edge, then $H''_s$ contains a matching between $A_{i'}$ and $B_{i}$ of size $30 \eps K m$.
	\item[\rm (a$_2$)] If $PS_s$ contains a $B_i A_{i'} $-edge, then $H''_s$ contains a matching between $A_{i+1}$ and $B_{i'-1}$ of size $30 \eps K  m$.
	\item[\rm (a$_3$)] $H''_1, \dots, H''_{|\mathcal{J} |}$ are edge-disjoint
	and for all $s\le |\mathcal{J} |$ the matchings guaranteed by {\rm (a$_1$)} and {\rm (a$_2$)} are edge-disjoint.
\end{itemize}
}
\noindent
So if $PS_s$ contains both an $A_i B_{i'} $-edge and a $B_{i'-1} A_{i+1} $-edge, then $H''_s$ contains a matching between $A_{i'}$ and $B_{i}$ of size  $60 \eps K m$.
\smallskip

\noindent
To prove the claim, first recall that $H''[A_i,B_{i'}]$ is a $248 \eps m/K$-regular graph for all $i,i' \le K$.
So $H''[A_{i},B_{i'}]$ can be decomposed into $248 \eps m/K$ perfect matchings. 
Each perfect matching can be split into $1/(31\eps K) $ matchings, each of size at least $30 \eps K  m$.
Therefore $H''[A_{i},B_{i'}]$ contains $8m/K^2$ edge-disjoint matchings, each of size at least $30 \eps K  m$.
(i) and ($\eps_s$) together imply that for any $i,i' \le K$, the number of $PS_s$ containing an $A_iB_{i'}$-edge is at most 
$$ \sum_{(\I)\ : \ i \in \{i_1,i_2\}, \ i' \in \{i_3,i_4\}} | \mathcal{J}_{\I} | \le 4 m/K^2.$$
Recall that $H''[A_{i'},B_{i}]$ contains $8m/K^2$ edge-disjoint matchings, each of size at least $30 \eps m$.
Thus we can assign a distinct matching in $H''[A_{i'},B_{i}]$ of size $30 \eps m$ to each $PS_s$ that contains an $A_{i}B_{i'}$-edge. 
Additionally, we can also assign a distinct matching in $H''[A_{i+1},B_{i'-1}]$ of size $30 \eps m$ to each $PS_s$ that contains a $B_{i}A_{i'}$-edge.%
    \COMMENT{Here we use that the number of $PS_s$ containing an $B_{i}A_{i'}$-edge is at most 
$\sum_{(i_2,i_3,i_4)\ : \ i \in \{i_3,i_4\}} | \mathcal{J}_{i',i_2,i_3,i_4} | \le 2 m/K^2.$}
For all $s \le |\mathcal{J}|$, let%
\COMMENT{Andy: set becomes let}
$H_s''$ be the union of all those matchings assigned to $PS_s$.%
    \COMMENT{Deryk: reworded}
Then $H_1'', \dots, H''_{|\mathcal{J} |}$ are as desired in the claim.
\medskip

For each $s\le |\mathcal{J} |$, we will now add suitable edges from $H''_s$ to $PS_s$ in order to obtain a path sequence $PS_s'$ which
is locally balanced with respect to $C$. So fix $s \le |\mathcal{J}|$ and let $e_1, \dots, e_{\ell}$ denote the edges of $PS_s$.
Note that $\ell = 2e(J^*_s) \le 6 K \eps m $ by~($\gamma_s$) and~(\ref{J^*}).%
    \COMMENT{Deryk: had $\ell = 2e(J^*_s) \le 2 \eps n \le 6 K \eps m $}
For each $r \le \ell$, we will find a directed edge $f_r$ satisfying the following conditions: 
\begin{itemize}
	\item[\rm (b$_1$)] If $e_r$ is an $A_i B_{i'} $-edge, then $f_r$ is an $A_{i'}B_{i}$-edge.
	\item[\rm (b$_2$)] If $e_r$ is a $B_i A_{i'} $-edge, then $f_r$ is a $B_{i'-1}A_{i+1}$-edge.
	\item[\rm (b$_3$)] The undirected version of $\{f_1, \dots, f_{\ell}\}$ is a matching in $H''_s$ and vertex-disjoint from $V(PS_s)$.
\end{itemize}
Suppose that for some $r \le \ell$ we have already constructed $f_1, \dots, f_{r-1}$.
Suppose that $e_r$ is an $A_i B_{i'} $-edge. (The argument for the other case is similar.) 
By~(a$_1$), $H''_s[A_{i'},B_{i}]$%
    \COMMENT{Daniela: had $H''_s[A_i,B_{i'}]$}
contains a matching of size $30 K \eps m $.
Note by ($\alpha_s$) and (b$_3$) that 
$$ |V(PS_s \cup \{f_1, \dots, f_{r-1} \})| \le 3 e(J^*_s) + 2(r-1) < 5 \ell \le 30 K \eps m. $$
Hence there exists an edge in $H''_s[A_{i'},B_{i}]$ that is vertex-disjoint from $PS_s \cup \{f_1, \dots, f_{r-1} \} $.
Orient one such edge from $A_{i'}$ to $B_{i}$ and call it $f_r$.
In this way, we can construct $f_1, \dots, f_{\ell}$ satisfying (b$_1$)--(b$_3$).

Let $PS'_s$ be digraph obtained from $PS_s$ by adding all the edges $f_1, \dots, f_{\ell}$.
Note that $PS'_s$ is a locally balanced path sequence with respect to~$C$.
(Indeed, $PS'_s$ is locally balanced since $\{e_r,f_r\}$ is locally balanced for each $r \le \ell$.)
Let $\I$ be such that $J\in\mathcal{J}_{\I}$. Then the following properties hold:
\begin{itemize}
	\item[(c$_1$)] $PS'_s$ is an $A_{i_1}$-extension of $J^*_{s,{\rm dir}}$.
	\item[(c$_2$)] $|V(PS_s') \cap A_i| , |V(PS_s') \cap B_i| \le 12 \eps K m$ for all $i\le K$.	
	\item[(c$_3$)] If $V(PS_s') \cap A_i \ne \emptyset$, then $i \in \{ \I, i_3+1 ,i_4+1 \}$.
	\item[(c$_4$)] If $V(PS_s') \cap B_i \ne \emptyset$, then $i \in \{ i_1-1, \I\}$.
\end{itemize}
Indeed, (c$_1$) is implied by ($\alpha_s$) and the definition of $PS'_s$.
Since $e(PS'_s) = 2 e(PS_s) = 4 e(J^*_s)$, (c$_2$) holds by~\eqref{J^*}.
Finally, (c$_3$) and (c$_4$) are implied by ($\eps_s$), (b$_1$) and (b$_2$) as $J_s \in \mathcal{J}_{\I}$.
 
Note that $PS'_1-J^*_{1,{\rm dir}}, \dots, PS'_{|\mathcal{J}|}-J^*_{|\mathcal{J}|,{\rm dir}}$ are%
     \COMMENT{Daniela: had $PS'_1, \dots, PS'_{|\mathcal{J}|}$}
pairwise edge-disjoint and let $\mathcal{BE}:=\{PS'_1, \dots, PS'_{|\mathcal{J}|}\}$.
We claim that $\mathcal{BE}$ is a balanced extension of $\mathcal{J}^*_{\rm dir}$ with respect to $(\mathcal{Q},C)$ and
parameters $( 12 \eps K, 12)$.
To see this, recall that $\mathcal{Q}=\{A_1,\dots,A_K,B_1,\dots,B_K\}$ is a $(2K,m)$-equipartition
of $V':=V\setminus (A_0\cup B_0)$. Clearly, (BE1) holds with $V'$ playing the role of~$V$.
(c$_3$) and (i) imply that for every $i \le K$ there are at most $6m/K$ $PS_s' \in \mathcal{BE}$ such that $V(PS_s') \cap A_i \ne \emptyset$.
A similar statement also holds for each~$B_i$.
So together with (c$_2$), this implies (BE3), where $2K$ plays the role of $k$ in (BE3).%
    \COMMENT{Need parameters $( 12 \eps K, 12)$ since we have $2K$ clusters and so we need that $6m/K=12m/2K$}
As remarked after the definition of a balanced extension, this implies the `moreover part' of (BE2). So (BE2) holds too.%
    \COMMENT{Daniela: previously had "Given $i \le K$, the number of $J$ such that $J \in \mathcal{J}_{i,i_2,i_3,i_4}$ for some $i_2,i_3,i_4$ is at most $m/K$ by~(i).
So (BE2) follows from (c$_1$), where $2K$ plays the role of $k$ in (BE2)."}
Therefore $\mathcal{BE}$ is a balanced extension, so the lemma follows (by orienting the remaining edges of $H$ arbitrarily).
\endproof

\removelastskip\penalty55\medskip\noindent{\bf Proof of Lemma~\ref{almostthmbip}. }
Apply Lemma~\ref{sysdecombip} to obtain tuples $(G_{j}, \mathcal{Q}, C_{j}, H_{j}, \mathcal{J}_{j})$ for all $j \le K/2$
satisfying (a$_1$)--(a$_6$). Fix $j \le K/2$ and write $\mathcal{J}_{j} :=\{J_{j,1}\dots,J_{j,|\mathcal{J}_{j}|}\}$.
Next, apply Lemma~\ref{balanceextensionbip} with $\mathcal{J}_{j},C_{j}, H_{j},  \eps_{0}$ playing the roles
of $\mathcal{J}, C, H, \eps$ to obtain an orientation $H_{j,{\rm dir}}$ of $H_j$ and a balanced extension $\mathcal{BE}_j$ of $\mathcal{J}_{j}$ with respect to
$(\mathcal{Q}, C_{j})$ and parameters $(12 \eps_0 K,12)$.
(Note that (a$_3$) and (a$_5$) imply conditions (i) and (iii) of Lemma~\ref{balanceextensionbip}.
Condition~(ii) follows from Lemma~\ref{almostthmbip}(d).)
So we can write $\mathcal{BE}_j :=\{PS_{j,1}\dots,PS_{j,|\mathcal{J}_{j}|}\}$ such that $(J_{j,s})^*_{\rm dir} \subseteq PS_{j,s}$ for all $s \le |\mathcal{J}_{j}|$.
Each path sequence in $\mathcal{BE}_j$ is obtained from some $(J_{j,s})^*_{\rm dir}$ by adding edges
of $H_{j,{\rm dir}}$.
Since $(G_{j, {\rm dir}}, \mathcal{Q}, C_{j})$ is a $(2K,m, 4 \mu, 5/K)$-cyclic system by Lemma~\ref{sysdecombip}(a$_6$), we can apply Lemma~\ref{merging} 
as follows:
\smallskip
\begin{center}
  \begin{tabular}{ r | c | c | c | c | c | c | c | c | c | c |c}
& $G_{j, {\rm dir}}$ & $\mathcal{Q}$ &  $C_{j}$  &  $2K$ & $\mathcal{J}^*_{j,{\rm dir}}$ & $|\mathcal{J}_j|$ & $4 \mu$ & $3\rho$ & $12\eps_{0}K$ & $5/K$ & $12$ \\ \hline
plays role of & $G$ & $\mathcal{Q}$ &  $C$ & $k$ & $\mathcal{M}$ & $q$ & $\mu$ & $\rho$ &  $\eps_0$ & $\eps'$ & $\ell$ \\
 \end{tabular}
\end{center}
\smallskip
\noindent
This gives us $|\mathcal{J}_{j}|$ directed Hamilton cycles $C'_{j,1}, \dots, C'_{j,|\mathcal{J}_{j}|}$ in $G_{j, {\rm dir}} + \mathcal{BE}_j$
such that each $C'_{j,s}$ contains $PS_{j,s}$ and is consistent with $(J_{j,s})^*_{\rm dir}$.
Moreover, (a$_4$) implies that%
     \COMMENT{Daniela: added (a$_4$) implies that}
$C'_{j,1} - (J_{j,1})^*_{\rm dir}, \dots, C'_{j,|\mathcal{J}_{j}|} - (J_{j,|\mathcal{J}_{j}|})^*_{\rm dir}$ are edge-disjoint subgraphs of $G_{j, {\rm dir}}+ H_{j, {\rm dir}}$.%
	\COMMENT{Allan: more details added.}
Repeat this process for all $j \le K/2$.

Recall from Lemma~\ref{sysdecombip}(a$_2$) that $\mathcal{J}_{1}, \dots, \mathcal{J}_{K/2}$ is a partition of $\mathcal{J}$.
Thus%
     \COMMENT{Daniela: had "Together with (a$_4$) this implies that" instead of "Thus" and "edge-disjoint directed Hamilton cycles" instead of
"directed Hamilton cycles" below}
we have obtained $|\mathcal{J}|$ directed Hamilton cycles 
$C'_{1}, \dots, C'_{|\mathcal{J}|}$ on $A\cup B$ such that each $C'_{s}$ is consistent with $(J_s)^*_{\rm dir}$ for some $J_s\in \mathcal{J}$
(and $J_s\neq J_{s'}$ whenever $s\neq s'$). 
Let $H_s$ be the undirected graph obtained from $C'_{s}-J_s^* +J_s$ by ignoring all the orientations of the edges.
Since $J_1, \dots, J_{|\mathcal{J}|}$ are edge-disjoint exceptional systems, $H_1, \dots, H_{|\mathcal{J}|}$ are edge-disjoint spanning subgraphs of $G$.
Finally, Proposition~\ref{CES-H2} implies that $H_1, \dots, H_{|\mathcal{J}|}$ are indeed as desired in Lemma~\ref{almostthmbip}.
\endproof


\medskip

{\footnotesize \obeylines \parindent=0pt
B\'ela Csaba
Bolyai Institute,
University of Szeged,
H-6720 Szeged, Aradi v\'ertan\'uk tere 1.
Hungary
\begin{flushleft}
\it{E-mail address}: \tt{bcsaba@math.u-szeged.hu}
\end{flushleft}

Daniela K\"{u}hn, Allan Lo, Deryk Osthus 
School of Mathematics
University of Birmingham
Edgbaston
Birmingham 
B15 2TT
UK
\begin{flushleft}
{\it{E-mail addresses}: \tt{\{d.kuhn,s.a.lo,d.osthus\}@bham.ac.uk}}
\end{flushleft}

Andrew Treglown
School of Mathematical Sciences
Queen Mary, University of London
Mile End Road
London 
E1 4NS
UK
\begin{flushleft}
{\it{E-mail address}: \tt{a.treglown}@qmul.ac.uk}
\end{flushleft}

}

\end{document}